\documentclass[12pt]{amsart}

\usepackage{ucs}

\usepackage{amssymb}
\usepackage{amsthm}
\usepackage{amsmath}
\usepackage{latexsym}
\usepackage[cp1251]{inputenc}
\usepackage{graphicx}
\usepackage{wrapfig}
\usepackage{caption}
\usepackage{subcaption}
\usepackage{indentfirst}
\usepackage[left=2.6cm,right=2.6cm,top=2.6cm,bottom=2.6cm,bindingoffset=0cm]{geometry}
\usepackage{enumerate}
\usepackage{makecell}
\usepackage{float}
\usepackage[normalem]{ulem}
\usepackage{color}

\DeclareMathOperator{\aut}{Aut}

\DeclareMathOperator{\cay}{Cay}
\DeclareMathOperator{\cyc}{Cyc}
\DeclareMathOperator{\dev}{Dev}

\DeclareMathOperator{\GL}{GL}
\DeclareMathOperator{\id}{id}

\DeclareMathOperator{\orb}{Orb}
\DeclareMathOperator{\out}{Out}

\DeclareMathOperator{\rk}{rk}

\DeclareMathOperator{\Span}{Span}

\DeclareMathOperator{\rad}{rad}

\DeclareMathOperator{\PL}{PL}
\DeclareMathOperator{\End}{End}

\def\tm#1{\item[{\rm (#1)}]}

\makeatletter 
\def\@seccntformat#1{\csname the#1\endcsname. } 
\def\@biblabel#1{#1.}

\makeatother

\title{Constructing linked systems of relative difference sets via Schur rings}

\author{Mikhail Muzychuk}
\address{Ben Gurion University of the Negev, Beer Sheva, Israel}
\email{muzychuk@bgu.ac.il}

\author{Grigory Ryabov}
\address{Ben Gurion University of the Negev, Beer Sheva, Israel}
\email{gric2ryabov@gmail.com}

\thanks{The second author is supported by The Israel Science Foundation (project No.~87792731)}

\date{}

\newtheorem{prop}{Proposition}[section]

\newtheorem{lemm}[prop]{Lemma}
\newtheorem{theo}[prop]{Theorem}

\newtheorem{corl}[prop]{Corollary}

\theoremstyle{definition}

\begin{document}

\maketitle

\begin{abstract}
In the present paper, we study relative difference sets (RDSs) and linked systems of them. It is shown that  a closed linked system of RDSs is always graded by a group. Based on this result, we also define a product of RDS linked systems sharing the same grading group. Further, we generalize the Davis-Polhill-Smith construction of a linked system of RDSs. Finally, we construct new linked system of RDSs in a Heisenberg group over a finite field and family of RDSs in an extraspecial $p$-group of exponent $p^2$. All constructions of new RDSs and their linked systems are based essentially on a usage of cyclotomic Schur rings.
\\
\\
\\
\textbf{Keywords}: relative difference sets, linked systems, Schur rings.
\\
\\
\\
\textbf{MSC}: 05B10, 20C05, 05E30. 
\end{abstract}

\section{Introduction}
 Linked systems of symmetric block designs were introduced by P. Cameron \cite{CS} as combinatorial objects associated to inequivalent 2-transitive representations of a given group. It turned out that these objects have close connections to association schemes, coding theory, and other parts of combinatorics. D. Higman generalized Cameron's idea and introduced uniformly linked strongly regular designs in \cite{H95}. A particular case of this object named linked system of divisible symmetric block designs was studied in details by H. Kharaghani and S. Suda  \cite{KhS}.

In 2014 J. Davis, W. Martin, and J. Polhill proposed to use difference sets to construct linked block designs -- they introduced a concept of linking\footnote{In what follows we prefer to use the word "linked" instead of "linking".} systems of difference sets~\cite{DMP} (see also~\cite{JLS}). A few years later this idea was extended to relative difference sets (RDSs) in~\cite{DPS1,DPS2}.  

In this paper we show how one can use Schur rings in order to construct systems of linked relative difference sets. As a result of this approach, we obtain two constructions of linked difference sets. 
The first one is based on cyclotomic Schur rings over the extraspecial groups of odd order. It provides  systems of linked relative
difference sets with new intersection numbers. The corresponding systems of group divisible block designs were unknown and marked in~\cite{KhS} as open cases.
The second construction generalizes the Davis-Polhill-Smith linked system~\cite{DPS1}.

To realize those constructions, we modified a well-known product concept
of two relative difference sets (Proposition 3.3). To our knowledge, this is the most general product defined so far.

The paper is organized as follows. Sections~\ref{srings} and~\ref{rds} contain all material about Schur rings and relative difference sets, respectively, needed in the paper. 

Section~4 deals with 
linked systems of RDSs. Only closed linked systems are considered. It is shown that there always exists a group 
which controls product of difference sets in the linked system. A product of two linked systems with the same associated group is defined in a subsection 4.3.

Section 5 presents a generalization of the Davis-Polhill-Smith construction~\cite{DPS1}.

The main result of Section 6 is formulated below. 
\begin{theo}\label{main2}
Let $q$ be an odd prime power and $r\geq 1$. Then there exists a closed linked system of semiregular relative difference sets with parameters~
$$(q^{2r},q,q^{2r},q^{2r-1},q,q^{2r-1}-q^r+q^{r-1},q^{2r-1}+q^{r-1})$$
in a Heisenberg group of dimension~$2r+1$ over $\mathbb{F}_q$.
\end{theo}

In Section 7, a semiregular relative difference set over the extraspecial group of exponent $p^2$ is described.
More precisely, we have the following

\begin{theo}\label{main1}
Let $p$ be an odd prime and $r\geq 1$. Then there exists a semiregular relative difference set with parameters~$(p^{2r},p,p^{2r},p^{2r-1})$ in an extraspecial group of order~$p^{2r+1}$ of exponent $p^2$.
\end{theo}

RDSs with the parameters~$(p^{2r},p,p^{2r},p^{2r-1})$ were studied in several papers (see, e.g.,~\cite{LM,MS,Schmidt}). Notice that the RDSs from Theorems~\ref{main2} and~\ref{main1} can not be obtained using the construction of RDSs over nonabelian $p$-groups from~\cite[Corollary~4.3]{MS}. In the case when an extraspecial group has order~$p^3$ and exponent~$p^2$, the forbidden subgroup in the RDS from Theorem~\ref{main1} can be nonnormal in the underlying group unlike the most known cases.

\section{$S$-rings}\label{srings}

In this section, we provide a necessary background for $S$-rings - a special class of group rings which goes back to the classical works of I.~Schur and H.~Wielandt~\cite{Schur,Wi}. We use the notations and terminology from~\cite[Section~2.4]{CP} and~\cite{Ry}, where the most part of the material is contained.

Let $G$ be a finite group and $\mathbb{Z}G$ the group ring over the integers. The identity element of $G$ is denoted by $e$. A product of two elements $x=\sum_{g\in G} x_g g,y=\sum_{g\in G} y_g g\in\mathbb{Z}G$ will be written as $x\cdot y$. Their scalar product $(x,y)$ is defined via $(x,y) = \sum_{g\in G} x_g y_g$.

Given $X\subseteq G$, we set $$\underline{X}=\sum \limits_{x\in X} {x}\in\mathbb{Z}G,~X^{(-1)}=\{x^{-1}:x\in X\},~X^\#=X\setminus\{e\}.$$
A subset $X$ of $G$ is called \emph{reversible} if $X=X^{(-1)}$. It is easy to check that if $X\subseteq H\leq G$, then $\underline{H}\cdot \underline{X}=\underline{X}\cdot \underline{H}=|X|\underline{H}$. 

A subring  $\mathcal{A}\subseteq \mathbb{Z} G$ is called an \emph{$S$-ring} (a \emph{Schur ring}) over $G$ if there exists a partition $\mathcal{S}=\mathcal{S}(\mathcal{A})$ of~$G$ such that:
\begin{enumerate}

\tm{1} $\{e\}\in\mathcal{S}$;

\tm{2} if $X\in\mathcal{S}$, then $X^{(-1)}\in\mathcal{S}$;

\tm{3} $\mathcal{A}=\Span_{\mathbb{Z}}\{\underline{X}:\ X\in\mathcal{S}\}$.
\end{enumerate}

\noindent  The elements of $\mathcal{S}$ are called the \emph{basic sets} of $\mathcal{A}$ and the number $\rk(\mathcal{A})=|\mathcal{S}|$ is called the \emph{rank} of~$\mathcal{A}$. It is easy to see that $\mathbb{Z}G$ is an $S$-ring. One more example of an $S$-ring is provided by the partition $\{\{e\},G^\#\}$. Clearly, the corresponding $S$-ring is of rank~$2$. One can verify that if $X,Y\in \mathcal{S}(\mathcal{A})$, then $XY\in \mathcal{S}(\mathcal{A})$ whenever $|X|=1$ or $|Y|=1$. We say that a subset $S\subseteq G$ is a \emph{fusion} of $\mathcal{A}$ if $S$ is a union of some basic sets of $\mathcal{A}$.

Given $X,Y,Z\in\mathcal{S}$, the number of distinct representations of $z\in Z$ as a product $z=xy$ with $x\in X$ and $y\in Y$ does not depend on the choice of $z\in Z$. Denote this number by $c^Z_{XY}$. One can see that $\underline{X}\cdot\underline{Y}=\sum_{Z\in \mathcal{S}(\mathcal{A})}c^Z_{XY}\underline{Z}$. Therefore the numbers  $c^Z_{XY}$ are the structure constants of $\mathcal{A}$ with respect to the basis $\{\underline{X}:\ X\in\mathcal{S}\}$.

%An $S$-ring $\mathcal{A}$ is said to be \emph{commutative} (\emph{symmetric}, resp.) if $c_{XY}^Z=c_{YX}^Z$ for all $X,Y,Z\in \mathcal{S}(\mathcal{A})$ ($X=X^{(-1)}$ for all $X\in \mathcal{S}(\mathcal{A})$, resp.). If $\mathcal{A}$ is symmetric, then $\mathcal{A}$ is also commutative. 

Let $K \leq \aut(G)$. Then the orbit partition $\orb(K,G)$ of  $G$ defines the $S$-ring over~$G$ which is called \emph{cyclotomic} and denoted by $\cyc(K,G)$.

An $S$-ring $\mathcal{A}$ is said to be \emph{amorphic} if for every partition $\mathcal{S}^\prime$ of $G$ such that $\{e\}\in \mathcal{S}^\prime$ and each set from $\mathcal{S}^\prime$ is a union of some sets from $\mathcal{S}$, the $\mathbb{Z}$-module $\Span_{\mathbb{Z}}\{\underline{X}:\ X\in\mathcal{S}^\prime\}$ is also an $S$-ring. An amorphic $S$-ring is a special case of an amorphic association scheme introduced in~\cite{GIK}. For details on amorphic association schemes, we refer the reader to~\cite{DM}. Clearly, every $S$-ring of rank at most~$3$ is amorphic.

\section{Relative difference sets}\label{rds}
This section contains a necessary background for relative difference sets (RDSs for short). The most of material here is taken from~\cite{BJL,Pott2}. A special attention is paid to connection of RDSs with partial difference sets (PDSs for short) and products of RDSs (see Subsections~$3.2$ and~$3.3$, respectively).

\subsection{Definitions}

Let $N$ be a subgroup of a group $G$. A subset $X$ of $G$ is called a \emph{relative to~$N$ difference set} in $G$ if 
$$\underline{X}\cdot\underline{X}^{(-1)}=ke+\lambda(\underline{G}-\underline{N}),$$
for $k=|X|$ and some positive integer $\lambda$. The numbers  $(m,n,k,\lambda)$, where $m=|G:N|$ and $n=|N|$, are called the \emph{parameters} of $X$ and the subgroup $N$ is called a \emph{forbidden subgroup}. Although in many concrete examples $N$ is normal in $G$, we do not suppose normality of $N$. A simple counting argument implies that $k(k-1)=\lambda n(m-1)$. Since nontrivial elements of $N$ do not appear in $\underline{X}\cdot \underline{X}^{(-1)}$, the intersection $X\cap Ng$ contains at most one element for each $g\in G$. If $|X\cap Ng|=1$ for every $g\in G$ or, equivalently, $X$ is a right transversal for $N$ in $G$, then $X$ is called \emph{semiregular}. In this case, $k=|X|=|G:N|=m=\lambda n$ and
$$\underline{N}\cdot\underline{X}=\underline{NX}=\underline{G}.$$
If, in addition, $N$ is normal in $G$, then we also have
$$\underline{X}\cdot\underline{N}=\underline{G}.$$
If $X$ is semiregular, then the parameters of $X$ can be expressed via $n$ and $\lambda$ as $(m,n,k,\lambda)=(\lambda n, n, \lambda n,\lambda)$. If $X$ is reversible semiregular RDS with forbidden subgroup~$N$, then $N\cap X=\{e\}$. An RDS $X$ is  called  \emph{symmetric}~\cite{Hir} if $X^{(-1)}$ is an RDS with the same parameters, but may admit a different forbidden subgroup. According to~\cite{Hir}, nonsymmetric RDSs exist. %For details on RDSs, we refer the reader to~\cite[Section~VI.10]{BJL}.

\subsection{Connection with partial difference sets}

Recall that a subset $X$ of $G$ is called a \emph{partial difference set} (\emph{PDS} for short) if 
$$\underline{X}\cdot\underline{X}^{(-1)}=|S|e+\lambda \underline{X}+\mu(\underline{G}^\#-\underline{X})$$
for some nonnegative integers $\lambda$ and $\mu$. The numbers $(v,k,\lambda,\mu)$, where $v=|G|$ and $k=|X|$, are called the \emph{parameters} of $X$. If $X$ is a reversible PDS, then the Cayley graph $\cay(G,X)$ is strongly regular. More details on PDSs can be found, e.g., in~\cite{Ma}. The following lemma provides a connection between RDS and PDS in some special case. In this lemma, we use the notions of a metric association scheme and antipodal distance regular cover which can be found in~\cite{DMM} and~~\cite{GH}, respectively.

\begin{lemm}\label{rdspds}
Let $e\in X\subseteq G$ be a reversible subset of a group $G$. Then the following statements are equivalent: 
\begin{enumerate}
	\tm{1} $X$ is a semiregular reversible $(n\lambda,n,n\lambda,\lambda)$-RDS in $G$ with a forbidden subgroup $N$;
	\tm{2} $\cay(G,X^\#)$ is an antipodal distance regular cover of a complete graph $K_{n\lambda}$ with parameters\footnote{We use here notation proposed in~\cite{GH}.} $(n\lambda,n,\lambda)$.
\end{enumerate}
If one of the above conditions holds and $\lambda=n$, then $S=X^\#\cup N^\#$ is a reversible PDS in $G$ with parameters $(n^3,n^2+n-2,n-2,n+2)$ and $\cay(G,S)$ is a strongly regular graph.
\end{lemm}

\begin{proof}
If $X$ satisfies~$(1)$, then  the $\mathbb{Z}$-submodule spanned by $e,\underline{X},\underline{N},\underline{G}$ is a $\mathbb{Z}$-subalgebra of $\mathbb{Z}G$. Now it follows from $\langle e,\underline{X}^\#,\underline{N}^\#,\underline{G\setminus (X\cup N)}\rangle=\langle e,\underline{X},\underline{N},\underline{G}\rangle$ that 
$$\langle e,\underline{X}^\#,\underline{N}^\#,\underline{G\setminus (X\cup N)}\rangle$$ 
is an $S$-ring over $G$. Put $X_0=\{e\}$, $X_1=X^\#$, $X_2=G\setminus(X\cup N)$, and $X_3=N^\#$. It follows from $\underline{X}^2=\lambda ne +\lambda (\underline{G}-\underline{N})$ and $\underline{X}\cdot\underline{N}=\underline{G}$ that $\langle\underline{X_i}\rangle_{i=0}^3$ is an $S$-ring for which $\underline{X_i}$ is a polynomial in $\underline{X_1}$ of degree $i$. Therefore the Cayley graphs $\cay(G,X_i)$ form a metric association scheme on the point set $G$ with $\cay(G,X_1)$ being a diameter~$3$ distance regular graph which is an antipodal cover of $K_{n\lambda}$ with parameters $(n\lambda,n,\lambda)$.

	If $X^\#$ satisfies~$(2)$, then the distance graphs $\Gamma_i=\{(u,v)\in G^2\,|\, \mathsf{dist}_{\cay(G,X^\#)}(u,v)=i\}$, $i=0,\ldots,3$, form a metric association scheme on $G$ with $\Gamma_1=\cay(G,X^\#)$. It follows from $\aut(\Gamma_1)\leq\aut(\Gamma_i)$ that  $\Gamma_i=\cay(G,X_i)$ for  certain $X_i\subseteq G$. Therefore $\langle\underline{X_i}\rangle_{i=0}^3$ is an $S$-ring over $G$. By antipodality of $\Gamma_1$, the set $N=X_3\cup\{e\}$ is a subgroup of $G$ of order $n$. By distance regularity of $\Gamma_1$, we obtain $\underline{X_1}^2=|X_1|e+b_1\underline{X}_1+c_2\underline{X_2}$. Together with $c_2=\lambda$, $|X_1|=n\lambda-1$, $|X_2|=|X_1|(n-1)$, we conclude that $b_1=\lambda-2$. Therefore
$$\underline{X}^2=\underline{X_1}^2+2\underline{X_1}+e=\lambda ne+\lambda(\underline{G}-\underline{N}),$$ 
as desired.   	

Assume now that $X$ satisfies~$(1)$ and $n=\lambda$. Then $S=X^\#\cup N^\#$ is reversible. An explicit computation in the group ring $\mathbb{Z}G$ yields that
$$\underline{S}^2=(\underline{S}+2e)^2-4\underline{S}-4e=(\underline{X}+\underline{N})^2-4\underline{S}-4e=\underline{X}^2+\underline{X}\cdot\underline{N}+\underline{N}\cdot\underline{X}+\underline{N}^2-4\underline{S}-4e=$$
$$=ke+\lambda (\underline{G}-\underline{N})+2\underline{G}+n\underline{N}-4\underline{S}-4e=(n+k-2)e+(n-2)\underline{S}+(n+2)(\underline{G}^\#-\underline{S}),$$
where the fourth equality holds by semiregularity of $X$, whereas the fifth one holds by $n=\lambda$. 
\end{proof}

A PDS is said to be of \emph{Latin square type} (\emph{negative Latin square type}, resp.) if it is reversible and identity-free and has parameters $(n^2,r(n-1),n+r^2-3r,r^2-r)$ ($(n^2,r(n+1),-n+r^2+3r,r^2+r)$, resp.) for some positive integers $n$ and $r$. From~\cite[Theorem~1]{DM} it follows that every nontrivial basic set of an amorphic $S$-ring of rank at least~$4$ is a PDS of Latin square or negative Latin square type. We say that an amorphic $S$-ring $\mathcal{A}$ is a \emph{Latin square type} if every nontrivial basic set of $\mathcal{A}$ is a PDS of Latin square type. %Clearly, an amorphic $S$-ring of Latin square type is symmetric and hence commutative. 

\subsection{Product of RDSs}

The product construction for RDSs was first introduced in~\cite[Theorem~2.1]{Dav} and was generalized later in~\cite[Result~2.4]{Schmidt}. The most general concept was proposed in~\cite[Lemma~3.1]{Hir}. In this subsection, we generalize the product proposed in~\cite{Hir} and provide an example which was not covered by the above product.

First, let us call $X\subseteq G$ \emph{inversely commuting} (\emph{i-commuting}, for short) if $\underline{X}\cdot \underline{X}^{(-1)}=\underline{X}^{(-1)}\cdot \underline{X}$.

\begin{lemm}\label{icommut}
Let $X\subseteq G$ be an RDS with parameters $(m,n,k,\lambda)$ and forbidden subgroup $N$. Then $X$ is i-commuting if and only if $\underline{X}\cdot\underline{N}=\underline{N}\cdot\underline{X}$.
\end{lemm}

\begin{proof}
Firstly,	let us prove the ``only if" part. On the one hand, we have
$$\underline{X}\cdot \underline{X}^{(-1)}\cdot \underline{X}=(ke+\lambda(\underline{G}-\underline{N}))\cdot \underline{X}=k\underline{X}+\lambda k \underline{G}-\underline{N}\cdot \underline{X}.$$
On the other hand, since $X$ is i-commuting,
$$\underline{X}\cdot \underline{X}^{(-1)}\cdot \underline{X}=\underline{X}\cdot(ke+\lambda(\underline{G}-\underline{N}))=k\underline{X}+\lambda k \underline{G}-\underline{X}\cdot\underline{N}.$$
Two above equalities imply that $\underline{X}\cdot\underline{N}=\underline{N}\cdot\underline{X}$ as required.	
	
Now let us prove the "if" part. It follows from  $\underline{N}\cdot\underline{X}=\underline{NX}$ that $NX$ is a union of $|X|$ disjoint right $N$-cosets. Now the assumption $\underline{X}\cdot\underline{N}=\underline{N}\cdot\underline{X}$ yields us that $XN=NX$. Therefore $x_1N\cap x_2N=\varnothing$ whenever $x_1\neq x_2\in X$. This implies that a nontrivial element of $N$ does not appear in the product $\underline{X}^{(-1)}\cdot\underline{X}$, i.e. 
	$\underline{X}^{(-1)}\cdot\underline{X}=k e+\sum_{x\in G\setminus N} \lambda_x x$. Clearly, $\sum_{x\in G\setminus N}\lambda_x = k^2-k=\lambda(|G|-|N|)$.
	\newline
	It follows from $(\underline{X}^{(-1)}\cdot\underline{X},\underline{X}^{(-1)}\cdot\underline{X}) = (\underline{X}\cdot\underline{X}^{(-1)},\underline{X}\cdot\underline{X}^{(-1)})$ that $$k^2+\sum_{x\in G\setminus N}\lambda_x^2 = k^2+\lambda^2(|G|-|N|)\Rightarrow \frac{1}{|G|-|N|}\sum_{x\in G\setminus N}\lambda_x^2 = \lambda^2 =\left(\frac{1}{|G|-|N|}\sum_{x\in G\setminus N}\lambda_x\right)^2.$$
	\newline
	The well known inequality between arithmetic and quadratic means implies
	that $\lambda_x=\lambda$ for all $x\in G\setminus N$. Finally, $\underline{X}^{(-1)} \underline{X} = k e+\lambda(\underline{G}-\underline{N})= \underline{X}\cdot\underline{X}^{(-1)}$.
\end{proof}

Notice that an i-commuting RDS $X$ is always symmetric, but we do not know the cases of symmetric non-i-commuting RDS. An RDS $X$ is always i-commuting in any of the following cases: $N\trianglelefteq G$; $X$ is a union of conjugacy classes of $G$; $X=X^{(-1)}$.

The statement below generalizes~\cite[Lemma 3.1]{Hir}. Our proof follows the one given in~\cite{Hir}.

\begin{prop}\label{product0}
Let $G=G_1G_2$ be a product of its proper subgroups $G_1$ and $G_2$ (not necessarily normal in $G$) and $N=G_1\cap G_2$. Suppose that $X_1$ and $X_2$ are semiregular RDSs in $G_1$ and $G_2$, respectively, with the same forbidden subgroup $N$ and parameters $(n\lambda_1,n,n\lambda_1,\lambda_1)$ and $(n\lambda_2,n,n\lambda_2,\lambda_2)$, respectively. If $X_1$ is i-commuting, then $X_1X_2$ is a semiregular RDS with forbidden subgroup $N$ and parameters $(n^2\lambda_1\lambda_2,n,n^2\lambda_1\lambda_2,n\lambda_1\lambda_2)$. Moreover, $X_1X_2$ is symmetric if and only if $X_2$ so is.
\end{prop}

\begin{proof}
Let $k_i=|X_i|=n\lambda_i$, $i\in \{1,2\}$. If $x_1 x_2 = x^\prime_1 x^\prime_2$ for some $x_i,x^\prime_i\in X_i$, then $\{e\} \cup (G_2\setminus N)\ni x^\prime_2x_2^{-1}=(x^\prime_1)^{-1} x_1 \in G_1$, implying $x_1=x_1^\prime,x_2=x^\prime_2$. Thus $X_1X_2$ has cardinality $|X_1||X_2|=k_1k_2=n^2\lambda_1\lambda_2$ and 
$\underline{X_1X_2}=\underline{X_1}\cdot\underline{X_2}$.
\newline
\indent
Now $$\underline{X_1X_2}\cdot \underline{(X_1X_2)}^{(-1)} =
\underline{X_1}\cdot\underline{X_2}\cdot \underline{X_2}^{(-1)}\cdot \underline{X_1}^{(-1)} = \underline{X_1} \left(k_2 e + \lambda_2 (\underline{G_2} - \underline{N})\right) \underline{X_1}^{(-1)}.$$
By Lemma~\ref{icommut} $\underline{X_1} \cdot\underline{N} =\underline{N}\cdot\underline{X_1}$. Therefore
$$\underline{X_1X_2}\cdot \underline{(X_1X_2)}^{(-1)} = k_2 \underline{X_1}\cdot\underline{X_1}^{(-1)} + 
\lambda_2\underline{X_1}\cdot\underline{G_2}\cdot
\underline{X_1}^{(-1)} -\lambda_2 \underline{N}\cdot\underline{X_1}\cdot\underline{X_1}^{(-1)} = $$
$$
=k_2(k_1 e + \lambda_1(\underline{G_1}-\underline{N})) + \lambda_2\underline{X_1}\cdot\underline{G_2}\cdot
\underline{X_1}^{(-1)} - \lambda_2 \underline{N}(k_1 e + \lambda_1(\underline{G_1} - \underline{N}))
= 
$$
$$
=k_1 k_2 e +(k_2\lambda_1 - \lambda_1\lambda_2 n )\underline{G_1} +(-k_1\lambda_2-k_2\lambda_1 + \lambda_1\lambda_2 n)\underline{N} + \lambda_2\underline{X_1}\cdot\underline{G_2}\cdot \underline{X_1}^{(-1)}.$$
Taking into account that $n\lambda_i=k_i$ we conclude that
$$
\underline{X_1X_2}\cdot \underline{(X_1X_2)}^{(-1)} =
k_1 k_2 e - \lambda_1\lambda_2 n\underline{N} + \lambda_2\underline{X_1}\cdot\underline{G_2}\cdot
\underline{X_1}^{(-1)}.
$$
Using the equalities
$\underline{X_1}\cdot\underline{G_2}=\frac{1}{n}\left(\underline{X_1}\cdot\underline{N}\cdot\underline{G_2}\right)=\frac{1}{n}\left(\underline{G_1}\cdot\underline{G_2}\right)=\underline{G_1G_2}$,
we finally obtain
$$
\underline{X_1X_2}\cdot \underline{(X_1X_2)}^{(-1)} =
k_1 k_2 e - \lambda_1\lambda_2 n\underline{N} + \lambda_2k_1\underline{G_1G_2}=k_1k_2 e + \lambda_1\lambda_2 n (\underline{G} - \underline{N}).
$$
\newline\indent
Let us compute the product $\underline{(X_1X_2)}^{(-1)} \cdot \underline{X_1X_2}$:
$$ \underline{(X_1X_2)}^{(-1)} \cdot \underline{X_1X_2} =
\underline{X_2}^{(-1)}\cdot \underline{X_1}^{(-1)}\cdot\underline{X_1}\cdot\underline{X_2}\cdot =\underline{X_2}^{(-1)} \cdot\left(k_1 e + \lambda_1 (\underline{G_1} - \underline{N})\right)\cdot \underline{X_2}=$$
$$
=k_1\underline{X_2}^{(-1)}\cdot\underline{X_2} + 
\lambda_1\underline{X_2}^{(-1)}\cdot\underline{G_1}\cdot\underline{X_2} - \lambda_1\underline{X_2}^{(-1)}\cdot\underline{G_2}=$$ 
$$=k_1\underline{X_2}^{(-1)}\cdot\underline{X_2} +  \lambda_1\underline{X_2}^{(-1)}\cdot\underline{G_1}\cdot\underline{X_2}-\lambda_1 k_2\underline{G_2}. 
$$
Taking into account the equalities
$\underline{G_1}\cdot\underline{X_2}=\frac{1}{n}\left(\underline{G_1}\cdot\underline{N}\cdot\underline{X_2}\right)=
\frac{1}{n}\left(\underline{G_1}\cdot\underline{G_2}\right)=\underline{G_1G_2}$,
we obtain 
$$
\underline{(X_1X_2)}^{(-1)} \cdot \underline{X_1X_2} = 
k_1\underline{X_2}^{(-1)}\cdot\underline{X_2} +  \lambda_1k_2\underline{G}-\lambda_1 k_2\underline{G_2}.
$$
This implies that $(X_1X_2)^{(-1)}$ is an RDS if and only if $(X_2)^{(-1)}$ is an $(n\lambda_2,n,n\lambda_2,\lambda_2)$-RDS. 
\end{proof}

Given a positive integer~$n$, a cyclic group of order~$n$ and a symmetric group of degree~$n$ are denoted by $C_n$ and $S_n$, respectively.

{\bf An example.} Consider the group $G = (C_9\rtimes C_6)\times S_3$ where the product $(C_9\rtimes C_6)$ is the holomorph of $C_9$. Denote by $A,B,C$ the natural subgroups of $G$ isomorphic to $C_9$, $C_6$, and $S_3$, respectively. Let $b\in B,c\in C$ be elements of order $3$. According to~\cite[Example~4.3]{Hir} there exists a $(12,3,12,4)$-RDS $X_2\subseteq G_2=BC$ with forbidden subgroup $N=\langle bc\rangle\cong C_3$. The subgroup $AN\cong C_9\rtimes C_3$ is extraspecial of exponent $9$. According to Proposition~\ref{p24}, the group $G_1=AN$ contains a reversible $(9,3,9,3)$-RDS $X_1\subseteq G_1$ with forbidden subgroup $N$. So $X_1$ is i-commuting and the assumptions of Proposition~\ref{product0} are satisfied. Therefore $X_1X_2$ is a $(108,3,108,36)$-RDS over the group $G=G_1 G_2$. It is not symmetric, because $X_2$ is not. It is not covered by the construction from~\cite{Hir}, since $N$ is not normal neither in $G_1$ nor $G_2$.

\section{Linked systems of RDSs}\label{linked}

\subsection{Definitions}

Let $N\leq G$ and $\mathcal{L}=\{X_{\alpha}:~\alpha\in S\}$ a collection of RDSs in $G$ with the same forbidden subgroup~$N$ and parameters~$(m,n,k,\lambda)$ indexed by the elements of a finite set $S$ with $|S|=s\geq 2$. The collection $\mathcal{L}$ is called a \emph{closed linked system} of RDSs relative to $N$  (see \cite[Definition 1.4]{DPS1} and the sentence thereafter\footnote{Although our definition of closedness is formally different form the one given in~\cite{DPS1}, it is equivalent to the original.}) if there exist a bijection $\chi=\chi_{\mathcal{L}}:S \rightarrow S$, a function $\psi=\psi_{\mathcal{L}}:(S\times S)\setminus D\rightarrow S$, where $D=\{(\alpha,\chi(\alpha))\in S\times S\}$, and nonnegative integers $\mu$ and $\nu$ such that $X_{\alpha}^{(-1)}=X_{\chi(\alpha)}$ for every $\alpha\in S$ and
\begin{equation}\label{linked}
\underline{X_\alpha}\cdot\underline{X_\beta}=
\begin{cases}
ke+\lambda(\underline{G\setminus N}),~\beta=\chi(\alpha),\\
\mu \underline{X_{\psi(\alpha,\beta)}}+\nu (\underline{G\setminus X_{\psi(\alpha,\beta)}}),~\beta\neq \chi(\alpha).
\end{cases}
\end{equation}
It is easy to see that $X_{\chi(\alpha)}^{(-1)}=X_{\alpha}$ and hence $\chi(\chi(\alpha))=\alpha$ for every $\alpha\in S$. We say that $(m,n,k,\lambda,s,\mu,\nu)$ are \emph{parameters} and $(\chi,\psi)$ a \emph{pair of characteristic functions} of $\mathcal{L}$. %Usually, we omit the numbers $\mu$ and $\nu$ in the list of parameters if they are not important. 
%A counting argument yields that 
%$$k^2=\mu k+\nu (mn-k).$$
%In particular, if each RDS from $\mathcal{L}$ is semiregular, we have
%$$n\lambda=\mu+\nu (n-1).$$
From~\cite[Proposition~3.5]{KhS} it follows that

$$\mu=\frac{1}{mn}\left(k^2\pm (mn-k)\sqrt{\frac{k(mn-k)}{m(n-1)}}\right)~\text{and}~\nu=\frac{k}{mn}\left(k\mp \sqrt{\frac{k(mn-k)}{m(n-1)}}\right).$$

If the RDSs $X_\alpha$ are semiregular, then $k=m$ and the above formula reads as follows:
\begin{equation}\label{munu-reg}
\mu=\frac{1}{n}\left(m\pm (n-1)\sqrt{m}\right)~\text{and}~\nu=\frac{1}{n}\left(m\mp \sqrt{m}\right).
\end{equation}
So, in this case there exists a positive integer $\ell$ with $m=\ell^2$ and at least one of the numbers $\ell^2+\ell,\ell^2-\ell$ is divisible by $n$.

Notice that the definition of a closed linked system given in~\cite{DPS1} does not contain characteristic functions defined above. 

In the present paper, we consider only closed linked systems. In general, it is possible to consider linked systems which are not necessarily closed. In this case every subset of a linked system is also a linked system. However, a subset of a closed linked system is not necessarily a closed linked system.

%The notion of a linked system of difference sets appeared in~\cite{DMP} and, more generally, a linked system of divisible designs going back to~\cite{CS} in~\cite{DMM}. For details on linked systems of difference sets and divisible designs, we refer the readers to~\cite{DMP,DPS1,JLS} and~\cite{DMM,KhS}, respectively. 

%As the author can checked, linked systems of RDS were considered first time in~\cite{DPS2}. Our definition is slightly different from the definition from~\cite{DPS2}. Firstly, we require that a linked system should be closed under taking the set of inverses which is not necessary in~\cite{DPS2}. Further, the definition from~\cite{DPS2} uses double subscripts for indexing of sets in the system which is equivalent to usage of single subscripts and (characteristic) function that maps a pair of indices to some index. Therefore under linked systems, we mean a special case of linked systems in the sense of~\cite{DPS2} and the constructed in the present paper linked system is a linked system according to definition from~\cite{DPS2}. We would like to mention that the notions of different types of isomorphisms of difference sets can be transferred in a natural way to linked systems of difference sets, however this is beyond the scope of this article.

\subsection{Group structure of a closed linked system}

In this subsection, we show that there exists a group structure behind a closed linked system of semiregular RDSs. More precisely, we have the following
\begin{prop}\label{function}
Let $G$ be a group and $\mathcal{L}=\{X_\alpha:~\alpha \in S\}$ a closed linked system of semiregular RDSs in $G$ with a forbidden subgroup $N$ and a pair of
characteristic functions~$(\chi,\psi)$. Then the set $S^\infty=S\cup \{\infty\}$ equipped with a unary operation $\widehat{\chi}$ and a binary operation $\widehat{\psi}$ defined as follows
$$\widehat{\chi}(\alpha)=
\begin{cases}
\chi(\alpha),~\alpha\neq \infty,\\
\infty,~\alpha=\infty,\\
\end{cases}$$ 
$$\widehat{\psi}(\alpha,\beta)=
\begin{cases}
\psi(\alpha,\beta),~\beta\neq \chi(\alpha),~\alpha,\beta\neq \infty,\\
\infty,~\beta=\widehat{\chi}(\alpha),\\
\alpha,~\alpha\neq\infty,~\beta=\infty,\\
\beta,~\alpha=\infty,~\beta\neq \infty,
\end{cases}$$ 
is a group, where $\infty$ is the identity element and $\widehat{\chi}(\alpha)$ is the element inverse to $\alpha$ for each $\alpha\in S^{\infty}$. If in addition, $\psi$ is commutative, then $S^\infty$ is abelian. 
\end{prop}

\begin{proof}
At first, we notice that each $X\in\mathcal{L}$ is i-commuting, since $X^{(-1)}\in\mathcal{L}$. Therefore $\underline{N}\cdot \underline{X}=\underline{X} \cdot \underline{N}=\underline{G}$ holds for every $X\in \mathcal{L}$.

Consider the factor-algebra $\mathfrak{A}=\mathbb{Q}G/\langle\underline{G}\rangle$. Denote by $x_\alpha$, $\alpha\in S$, the image of $\underline{X}_\alpha$ in $\mathfrak{A}$ and by $x_\infty$ the image of $ke - \lambda\underline{N}$ in the same algebra. Consider the set $\mathfrak{S}=\{\langle x_\alpha\rangle\,|\,\alpha\in S^\infty\}$. It follows from~\eqref{linked} that $\langle x_\alpha\rangle \langle x_\beta \rangle = \langle x_{\widehat{\psi}(\alpha,\beta)}\rangle$ for each $\alpha,\beta\in S$. Taking into account the equalities
$$
(ke -\lambda\underline{N})\underline{X}_\alpha = k\underline{X}_\alpha -\lambda\underline{G} = \underline{X}_\alpha(ke -\lambda\underline{N}),$$ 
we obtain 
that $\langle x_\infty\rangle \langle x_\alpha \rangle = \langle x_\alpha \rangle = \langle x_{\widehat{\psi}(\infty,\alpha)}\rangle $ and 
$\langle x_\alpha \rangle \langle x_\infty\rangle= \langle x_\alpha \rangle = \langle x_{\widehat{\psi}(\alpha,\infty)}\rangle$. Combining altogether, we conclude that the set  $\mathfrak{S}$ is closed with respect to the multiplication of $\mathfrak{A}$. Since the product in $\mathfrak{A}$ is associative, its restriction on $\mathfrak{S}$ is associative too. Therefore $\widehat{\psi}$ is an associative operation on $S^\infty$ with a unit $\infty$ and an inversion $\alpha\mapsto\widehat{\chi}(\alpha)$, i.e. $S^\infty$ is a group.
\end{proof}

Given a closed linked system of semiregular RDSs $\mathcal{L}=\{X_\alpha:~\alpha\in S\}$, the group $S^\infty=S^\infty(\mathcal{L})$ from Proposition~\ref{function} is called a group \emph{associated} with $\mathcal{L}$. Clearly, two closed linked systems consisting of RDSs indexed by the elements of the same set $S$ with the same pairs of characteristic functions have the same associated groups.

\subsection{Product of linked systems}

In the proposition below, we give a construction of a product of closed linked systems of relative difference sets.

\begin{prop}\label{product}
Let $G=G_1G_2$ be a central product of its proper subgroups $G_1$ and $G_2$ and $Z=G_1\cap G_2\leq Z(G)$. Suppose that $\mathcal{L}_1=\{X_\alpha:~\alpha\in S\}$ and $\mathcal{L}_2=\{Y_\alpha:~\alpha\in S\}$, where $|S|=s\geq 2$, are closed linked systems of relative to~$Z$ semiregular difference sets in $G_1$ and $G_2$, respectively, with parameters $(n\lambda_1,n,n\lambda_1,\lambda_1,s,\mu_1,\nu_1)$ and $(n\lambda_2,n,n\lambda_2,\lambda_2,s,\mu_2,\nu_2)$, respectively, and the same pair of characteristic functions $(\chi,\psi)$. Then for every $f\in\aut(S^\infty)$, the set 
$$\mathcal{L}=\{X_\alpha Y_{f(\alpha)}:~\alpha\in S\}$$ 
is a closed linked system of relative to~$Z$ semiregular difference sets in~$G$ with parameters 
$$(n^2\lambda_1\lambda_2,n,n^2\lambda_1\lambda_2,n\lambda_1\lambda_2,s,\mu,\nu),$$ 
where
$$\mu=\mu_1\mu_2+(n-1)\nu_1\nu_2~\text{and}~\nu=\mu_1\nu_2+\mu_2\nu_1+(n-2)\nu_1\nu_2,$$ 
and $(\chi_{\mathcal{L}},\psi_{\mathcal{L}})=(\chi,\psi)$.
\end{prop}

\begin{proof}
%Firstly, let us prove Statement~$(1)$. It is easy to see that $|G|=|G_1||G_2|/|Z|=nm_1m_2$ and $|G:Z|=m_1m_2$. Since $X$ and $Y$ are semiregular, $|XY|=k_1k_2$. One can compute that
%$$\underline{XY}(\underline{XY})^{-1}=(\underline{XX})^{-1}(\underline{YY})^{-1}=(k_1e+\lambda_1(\underline{G_1\setminus Z}))(k_2e+\lambda_2(\underline{G_2\setminus Z}))=$$
%$$=k_1k_2e+k_1\lambda_2(\underline{G_2}-\underline{Z})+k_2\lambda_1(\underline{G_1}-\underline{Z})+\lambda_1\lambda_2(\underline{G_1}-\underline{Z})(\underline{G_2}-\underline{Z})=$$
%$$=k_1k_2e+n\lambda_1\lambda_2(\underline{G_1}+\underline{G_2}-2\underline{Z})+n\lambda_1\lambda_2(\underline{G_1G_2}-\underline{G_1}-\underline{G_2}+\underline{Z})=$$
%$$=k_1k_2e+n\lambda_1\lambda_2(\underline{G}-\underline{Z}),$$
%where the first equality holds by $[G_1,G_2]=\{e\}$ and the fourth equality holds because $X$ and $Y$ are semiregular. Thus, $XY$ is a relative to~$Z$ difference set in~$G$ with parameters~$(m_1m_2,n,k_1k_2,n\lambda_1\lambda_2)$. Observe that $k_1k_2=(n\lambda_1)(n\lambda_2)=n(n\lambda_1\lambda_2)$ and hence $XY$ is semiregular.

Each set from $\mathcal{L}$ is a relative to~$Z$ semiregular difference set with the required parameters by Proposition~\ref{product0}. One can see that 
$$(X_\alpha Y_{f(\alpha)})^{(-1)}=(X_\alpha)^{(-1)} (Y_{f(\alpha)})^{(-1)}=X_{\chi(\alpha)} Y_{\chi(f(\alpha))}=X_{\chi(\alpha)} Y_{f(\chi(\alpha))}\in \mathcal{L}$$
for all $\alpha\in S$, where the third equality holds because $f\in \aut(S^\infty)$. Next, we will verify that~\eqref{linked} holds for $\mathcal{L}$. Suppose that $X_\alpha Y_{f(\alpha)},X_\beta Y_{f(\beta)}\in \mathcal{L}$. If $\beta=\chi(\alpha)$, then we are done by Proposition~\ref{product0}. 

Suppose that $\beta\neq \chi(\alpha)$. Before we will compute the product $\underline{X_\alpha Y_{f(\alpha)}}\cdot\underline{X_\beta Y_{f(\beta)}}$, we make an auxiliary computation. Namely, we verify that
\begin{equation}\label{aux}
\underline{G_1}\cdot\underline{Y_\gamma}=\underline{Y_\gamma}\cdot\underline{G_1}=\underline{X_\gamma}\cdot\underline{G_2}=\underline{G_2}\cdot \underline{X_\gamma}=\underline{G}
\end{equation}
for every $\gamma\in S$. Since $Y_\gamma$ is semiregular, it is a transversal for $Z$ in $G_2$. So $ZY_\gamma=Y_\gamma Z=G_2$. Analogously, $ZX_\gamma=X_\gamma Z=G_1$. This implies that 
$$G_1Y_\gamma=G_1 Z Y_\gamma= G_1 G_2=G~\text{and}~X_\gamma G_2=X_\gamma Z G_2=G_1G_2=G.$$
Observe that each of the elements $\underline{G_1}\cdot\underline{Y_\gamma}$ and $\underline{X_\gamma}\cdot\underline{G_2}$ is a sum of $|G|$ elements of $G$ and hence $\underline{G_1}\cdot\underline{Y_\gamma}=\underline{X_\gamma}\cdot\underline{G_2}=\underline{G}$. The remaining equalities from~\eqref{aux} follow from the same argument and $[G_1,G_2]=\{e\}$.

Let us return to the computation of $\underline{X_\alpha Y_{f(\alpha)}}\cdot\underline{X_\beta Y_{f(\beta)}}$:
$$\underline{X_\alpha Y_{f(\alpha)}}\cdot\underline{X_\beta Y_{f(\beta)}}=\left(\underline{X_\alpha}\cdot \underline{X_\beta}\right) \cdot\left(\underline{Y_{f(\alpha)}}\cdot\underline{Y_{f(\beta)}}\right)=$$
$$=\left((\mu_1-\nu_1)\underline{X_{\psi(\alpha,\beta)}}+\nu_1\underline{G}_1\right)\cdot\left((\mu_2-\nu_2)\underline{Y_{\psi(f(\alpha),f(\beta))}}+\nu_2\underline{G}_2\right)=$$
$$=(n\nu_1\nu_2+\nu_1(\mu_2-\nu_2)+\nu_2(\mu_1-\nu_1))\underline{G}+(\mu_1-\nu_1)(\mu_2-\nu_2)\underline{X_{\psi(\alpha,\beta)}}\cdot\underline{Y_{\psi(f(\alpha),f(\beta))}}=$$
$$=(\mu_1\mu_2+(n-1)\nu_1\nu_2)\underline{X_{\psi(\alpha,\beta)}Y_{f(\psi(\alpha,\beta))}}+(\mu_1\nu_2+\mu_2\nu_1+(n-2)\nu_1\nu_2)(\underline{G}-\underline{X_{\psi(\alpha,\beta)}Y_{f(\psi(\alpha,\beta))}}),$$
where the first, third, and fourth  equalities hold by $[G_1,G_2]=\{e\}$,~\eqref{aux}, and $f\in \aut(S^\infty)$, respectively. Thus,~\eqref{linked} holds for $\mathcal{L}$, $\chi$, $\psi$, $\lambda=n\lambda_1\lambda_2$, $\mu=\mu_1\mu_2+(n-1)\nu_1\nu_2$, and $\nu=\mu_1\nu_2+\mu_2\nu_1+(n-2)\nu_1\nu_2$ as required.
\end{proof}

\section{Generalization of Davis-Polhill-Smith construction}
In this section, we generalize a construction of a linked system of RDSs given in~\cite{DPS2}. Let $n\geq 2$, $G$ a group of order~$n^2$, $t\geq 2$ a divisor of $n$, $H$ an abelian group of order~$t$, and $\mathcal{A}$ an amorphic $S$-ring of Latin square type of rank~$t+1$ over $G$ whose nontrivial basic sets $X_h$ are indexed by the elements of $H$. Suppose that $|X_h|=k_h(n-1)$, where $k_e=\frac{n}{t}+1$ and $k_h=\frac{n}{t}$ for $h\in H^\#$. Then by~\cite[Corollary~1]{DM}, we have
\begin{equation}\label{amorph}
\underline{X_h}\cdot \underline{X_{h^\prime}}=
\begin{cases}
k_h(n-1)e+(n-2k_h)\underline{X_h}+k_h(k_h-1)\underline{G}^\#,~h^\prime=h,\\
k_hk_{h^\prime}\underline{G}^\#-k_{h^\prime} \underline{X_h}-k_h\underline{X_{h^\prime}},~h^\prime\neq h
\end{cases}
\end{equation}
for all $h,h^\prime\in H$.

Due to~\cite{GIK}, for every prime power $n$, there exists an amorphic $S$-ring of Latin square type and rank~$n+2$ over an elementary abelian group $G$ of order~$n^2$. The $S$-ring described above has $n+1$ nontrivial basic sets $S_1,...,S_{n+1}$ of equal cardinality $n-1$. Let $t$ be a divisor of $n$ and $H$ an abelian group of order $t$. Take an arbitrary partition $\{P_h\}_{h\in H}$ of $\{1,...,n+1\}$ into $t$ classes labelled by the elements of $H$  such that $|P_{e_H}|=\frac{n}{t} + 1$ and $|P_h|=\frac{n}{t}, h\neq e_H$. Then the sets $\{e_G\}$, $X_h=\bigcup_{i\in P_h} S_i$ form an amorphic $S$-ring over satisfying~\eqref{amorph}. Thus, we obtain the next statement.

\begin{lemm}\label{amorphq}
For every prime power $n$ and a divisor $t$ of $n$, there exists an amorphic $S$-ring of Latin square type and rank~$t+1$ satisfying~\eqref{amorph}  over an elementary abelian group of order~$n^2$.
\end{lemm}

It is well-known that there exists a natural abelian group structure on the set of all endomorphisms $\mathrm{End}(H)$ of an abelian group $H$. It will be convenient to use multiplicative notation for group operation in $H$ and $+$ for addition of endomorphisms in $\mathrm{End}(H)$. To make notation coherent, we write $h^f$ for the image of $h\in H$ under an endomorphism $f\in \mathrm{End}(H)$. Thus, $h^f h^g = h^{f+g}$. We denote the trivial endomorphism as $\mathbf{o}$ and write $-f$ for an additive inverse to $f$. Notice that $\aut(H)$ is the multiplicative group of $\End(H)$.

Given $f\in \aut(H)$, let us define the set $Y_f\subseteq H\times G$ as follows:
$$Y_f=\{e\}\cup \bigcup \limits_{h\in H} h^f X_h.$$
Since each $X_h$ is reversible, 
$$Y_f^{(-1)}=Y_{-f}.$$
 One can see that 
$$|Y_f|=1+(\frac{n}{t}+1)(n-1)+(t-1)\frac{n}{t}(n-1)=n^2.$$
By the definition, $Y_f$ is a transversal for $H$ in $H\times G$ and hence $\underline{Y_f}\cdot \underline{H}=\underline{H}\cdot \underline{Y_f}=\underline{H\times G}$.

\begin{prop}\label{davis}
Given $f_1,f_2\in \aut(H)$, 
\begin{equation}\label{davis1}
\underline{Y_{f_1}}\cdot\underline{Y_{f_2}}=
\begin{cases}
n^2e+\frac{n^2}{t}\underline{H\times G}^\#,~f_2=-f_1,\\
n\underline{Y_{f_1 + f_2}}+(n-1)\frac{n}{t}\underline{G\times H},~f_1 + f_2\in \aut(H).
\end{cases}
\end{equation}
\end{prop}

\begin{proof}
An explicit computation in the group ring of $\mathbb{Z}(H\times G)$ using~\eqref{amorph} implies that
$$\underline{Y_{f_1}}\cdot \underline{Y_{f_2}}=\left(e+\sum\limits_{h\in H} h^{f_1}\underline{X_h}\right)\cdot \left(e+\sum\limits_{h\in H} h^{f_2}\underline{X_h}\right)=$$
$$=e+\sum\limits_{h\in H} h^{f_1}\underline{X_h}+\sum\limits_{h\in H} h^{f_2}\underline{X_h}+\sum\limits_{h\in H} h^{f_1}h^{f_2}\underline{X_h}^2+\sum\limits_{h\neq h^\prime\in H} h^{f_1}{h^\prime}^{f_2}\underline{X_h}\cdot\underline{X_{h^\prime}}=$$
$$=e+\sum\limits_{h\in H} h^{f_1}\underline{X_h}+\sum\limits_{h\in H} h^{f_2}\underline{X_h}+\sum\limits_{h\in H} h^{f_1}h^{f_2}(k_h(n-1)e+(n-2k_h)\underline{X_h}+k_h(k_h-1)\underline{G}^\#)+$$
$$+\sum\limits_{h\neq h^\prime\in H} h^{f_1}{h^\prime}^{f_2}(k_hk_{h^\prime}\underline{G}^\#-k_{h^\prime}\underline{X_h}-k_h\underline{X_{h^\prime}})=$$
$$=e+T_1+T_2+T_3\cdot\underline{G}^\#-T_4-T_5,$$
where
$$T_1=\sum\limits_{h\in H} (h^{f_1}+h^{f_2}+(n-2k_h)h^{f_1}h^{f_2})\underline{X_h},$$
$$T_2=(n-1)\sum\limits_{h\in H} k_hh^{f_1}h^{f_2},$$
$$T_3=\left(\sum\limits_{h\in H} k_h(k_h-1)h^{f_1}h^{f_2}+\sum\limits_{h\neq h^\prime\in H} h^{f_1}{h^\prime}^{f_2}k_hk_{h^\prime}\right),$$
$$T_4=\sum\limits_{h\neq h^\prime\in H} k_{h^\prime}h^{f_1}{h^\prime}^{f_2}\underline{X_h},$$
and
$$T_5=\sum\limits_{h\neq h^\prime\in H} k_hh^{f_1}{h^\prime}^{f_2}\underline{X_{h^\prime}}.$$
Let us compute each $T_i$. Firstly,
$$T_1=\underline{Y_{f_1}}^\#+\underline{Y_{f_2}}^\#+\sum\limits_{h\in H}(n-2k_h)h^{f_1}h^{f_2}\underline{X_h}=$$
$$=\begin{cases}
\underline{Y_{f_1}}^\#+\underline{Y_{f_2}}^\#+(n-\frac{2n}{t})\underline{G}^\#-2\underline{X_e},~f_2=-f_1,\\
\underline{Y_{f_1}}^\#+\underline{Y_{f_2}}^\#+(n-\frac{2n}{t})\underline{Y^\#_{f_1+f_2}}-2\underline{X_e},~f_1 + f_2\in \aut(H).
\end{cases}$$ 
If $f_1 + f_2\in \aut(H)$, then $\sum\limits_{h\in H} h^{f_1}h^{f_2}=\sum\limits_{h\in H} h^{f_1+f_2} = \underline{H}$. So
$$T_2=\begin{cases}
n^2-1,~f_2=-f_1,\\
(n-1)(\frac{n}{t}\underline{H}+e),~f_1 + f_2\in \aut(H).
\end{cases}$$
Since $f_1,f_2\in \aut(H)$, we have $\sum\limits_{h\in H} h^{f_1}=\sum\limits_{h\in H} h^{f_2}=\underline{H}$. Using this, we obtain
$$T_3=\left(\sum\limits_{h\in H} k_hh^{f_1}\right)\cdot\left(\sum\limits_{h\in H} k_hh^{f_2}\right)-\sum\limits_{h\in H} k_hh^{f_1}h^{f_2}=(\frac{n}{t}\underline{H}+e)^2-\sum\limits_{h\in H} k_hh^{f_1}h^{f_2}=$$
$$=\begin{cases}
(\frac{n^2}{t}+2\frac{n}{t})\underline{H}-ne,~f_2=-f_1,\\
(\frac{n^2}{t}+\frac{n}{t})\underline{H},~f_1 + f_2\in \aut(H).
\end{cases}$$
Note that $\underline{Y_{f_1}}^\#\cdot \underline{H}=\underline{Y_{f_2}}^\#\cdot \underline{H}=\underline{H\times G}^\#$ because $Y_{f_1}$ and $Y_{f_2}$ are transversals for $H$ in $H\times G$. Using this, one can compute $T_4$:
$$T_4=\sum\limits_{h^\prime\in H} k_{h^\prime} {h^\prime}^{f_2} (\underline{Y_{f_1}}^\#-{h^\prime}^{f_1}\underline{X_{h^\prime}})=\underline{Y_{f_1}}^\#(\frac{n}{t}\underline{H}+e)-\sum\limits_{h^\prime\in H} k_{h^\prime}{h^\prime}^{f_1}{h^\prime}^{f_2}\underline{X_{h^\prime}}=$$
$$=\begin{cases}
\frac{n}{t}\underline{H}^\#\cdot \underline{G}^\#+\underline{Y_{f_1}}^\#-\underline{X_e},~f_2=-f_1,\\
\frac{n}{t}\underline{H\times G}^\#+\underline{Y_{f_1}}^\#-\frac{n}{t}\underline{Y^\#_{f_1 + f_2}}-\underline{X_e},~f_1 + f_2\in \aut(H).
\end{cases}$$
Similarly,
$$T_5=\begin{cases}
\frac{n}{t}\underline{H}^\#\cdot \underline{G}^\#+\underline{Y_{f_2}}^\#-\underline{X_e},~f_2=-f_1,\\
\frac{n}{t}\underline{H\times G}^\#+\underline{Y_{f_2}}^\#-\frac{n}{t}\underline{Y^\#_{f_1 + f_2}}-\underline{X_e},~f_1 + f_2\in \aut(H).
\end{cases}$$
Substituting expressions for $T_1$, $T_2$, $T_3$, $T_4$, and $T_5$ to $\underline{Y_{f_1}}\cdot \underline{Y_{f_2}}=e+T_1+T_2+T_3\cdot\underline{G}^\#-T_4-T_5$, we obtain~\eqref{davis1} as required.
\end{proof}

The corollary below immediately follows from definitions and Proposition~\ref{davis}.

\begin{corl}\label{linkeddavis}
Let $S\leq \End(H)$ such that $S^\#\subseteq \aut(H)$. Then $\mathcal{L}=\{Y_f:~f\in S^\#\}$ is a closed linked system of RDSs in $H\times G$ with forbidden subgroup $H$, parameters
$$(n^2,t,n^2,\frac{n^2}{t},|S|-1,n+(n-1)\frac{n}{t},(n-1)\frac{n}{t}),$$
a pair of characteristic functions 
$$\chi(f)=-f~\text{and}~\psi(f_1,f_2)=f_1 + f_2,$$
and associated group isomorphic to~$S$.
\end{corl}

The lemma below describes all possible $H$ and $S$ from Corollary~\ref{linkeddavis}.

\begin{lemm}\label{abgroup}
The following statements hold.
\begin{enumerate}
\tm{1} If $S\leq \End(H)$ is such that $S^\#\subseteq \aut(H)$, then $S\cong C_p^i$ and $H\cong C_p^j$ for some prime $p$ and positive integers $i\leq j$.
\tm{2} If $H\cong C_p^j$ for some prime $p$ and $j\geq 1$, then for every $i\leq j$, there is $S\leq \End(H)$ such that $S^\#\subseteq \aut(H)$ and $S\cong C_p^i$.
\end{enumerate}
\end{lemm}

\begin{proof}
$(1)$ Let $p$ be a prime divisor of $|H|$. For every $f\in S^\#$, we have $pf\in S\setminus \aut(H)$. So $pf=\textbf{o}$ and hence $S\cong C_p^i$ for some $i\geq 1$. Since $f\in \aut(H)$ and $(h^f)^p = h^{pf}=h^{\textbf{o}}=e$ for every $h\in H$, we obtain $h^p=e$ for every $h\in H$. Therefore $H\cong C_p^j$ for some $j\geq 1$.

Pick an arbitrary $h\in H^\#$. Then $h^S=\{h^f:~f\in S\}$ is a subgroup of $H$. So $|h^S|$ divides $|H|$. Observe that if $h^f=e$ for some $f\in S$, then $f=\textbf{o}$. Therefore the homomorphism $f\mapsto h^f$ from $S$ to $h^S$ is a monomorphism. Thus, $p^i=|S|=|h^S|$ divides $|H|=p^j$.
%and Statement~$(1)$ holds. 

$(2)$ From~\cite[Lemma~3.2]{Gur} it follows that the vector space $\End(C_p^j)\cong M_j(\mathbb{Z}_p)$ contains a $j$-dimensional subspace in which every nonzero matrix is invertible. So if $H\cong C_p^j$, then $\End(H)$ has a subgroup of order $p^j$ whose all nontrivial elements belong to $\aut(H)$.
%and hence Statement~$(2)$ holds.
\end{proof}

Now Lemma~\ref{amorphq}, Corollary~\ref{linkeddavis}, and Lemma~\ref{abgroup} imply the following statement.

\begin{corl}\label{linkedgen}
For every prime power $n$, divisor $t$ of $n$, and divisor $s$ of $t$, there is a closed linked system of RDSs with parameters 
$$(n^2,t,n^2,\frac{n^2}{t},s-1,n+(n-1)\frac{n}{t},(n-1)\frac{n}{t})$$
in an elementary abelian group of order~$n^2t$.
\end{corl}

The linked system from~\cite{DPS2} is a special case of the above linked system when $n=q^r$ and $t=s=q$ for a prime power $q$ and positive integer $r$.

\section{Linked system of RDSs in a Heisenberg group}\label{heisenberg}

In this section, we prove a statement (see Corollary~\ref{heis2r}) which is refinement of Theorem~\ref{main2}. Recall that each Heisenberg group is a central product of Heisenberg groups of dimension~$3$. So in view of Proposition~\ref{product} and~\eqref{munu-reg}, to prove Theorem~\ref{main2}, it is enough to construct a closed linked system of RDSs with parameters~$(q^2,q,q^2,q,q)$ in a Heisenberg group of dimension~$3$. To do this, we study a special cyclotimic $S$-ring over this group and show that each basic set of this $S$-ring outside the center together with the identity element forms a relative to the center difference set (Corollary~\ref{heissystem}).

Let $q$ be an odd prime power and $\mathbb{F}_q$ a finite field of order~$q$. Throughout this section except Corollary~\ref{heis2r}, 
$$G=\left\{ 
\left(\begin{smallmatrix}
1 & x & z\\
0 & 1 & y\\
0 & 0 & 1
\end{smallmatrix}\right):~x,y,z\in \mathbb{F}_q \right\}$$
is a Heisenberg group of dimension~$3$ over $\mathbb{F}_q$. Put
$$Z=Z(G)=\left\{ 
\left(\begin{smallmatrix}
1 & 0 & z\\
0 & 1 & 0\\
0 & 0 & 1
\end{smallmatrix}\right):~z\in \mathbb{F}_q \right\}.$$
Clearly, $|G|=q^3$ and $|Z|=q$.

Let $\varepsilon\in \mathbb{F}_q$ be a nonsquare and 
$$\mathcal{M}=\mathcal{M}(\varepsilon)=\left\{ 
\left(\begin{smallmatrix}
\alpha & \beta \\
\varepsilon \beta & \alpha \\
\end{smallmatrix}\right):~\alpha,\beta\in \mathbb{F}_q,~(\alpha,\beta)\neq (0,0) \right\}.$$
It is easy to check (see, e.g.,~\cite[Chapter~2.5]{LN}) that $\mathcal{M}$ (with the standard matrix multiplication) is a subgroup of $\GL_2(q)$ isomorphic to the multiplicative group $\mathbb{F}_{q^2}^\times$. In particular, $|\mathcal{M}|=q^2-1$ and $\mathcal{M}$ is cyclic. Given $M=\left(\begin{smallmatrix}
\alpha & \beta \\
\varepsilon \beta & \alpha \\
\end{smallmatrix}\right)\in \mathcal{M}$, let us define $\varphi=\varphi(M):G\rightarrow G$ as follows:
$$\varphi\left(\left(\begin{smallmatrix}
1 & x & z\\
0 & 1 & y\\
0 & 0 & 1
\end{smallmatrix}\right)\right)=
\left(\begin{smallmatrix}
1 & \alpha x+\varepsilon \beta y & F_{\alpha,\beta}(x,y,z)\\
0 & 1 & \beta x+ \alpha y\\
0 & 0 & 1
\end{smallmatrix}\right),$$
where $F_{\alpha,\beta}(x,y,z)=\alpha\beta(\frac{x^2}{2}+\varepsilon\frac{y^2}{2})+\varepsilon \beta^2 xy +(\alpha^2-\varepsilon \beta^2)z$.

\begin{lemm}\label{autheis}
The mapping $M\mapsto \varphi(M)$ is a monomorphism from $\mathcal{M}$ to $\aut(G)$.
\end{lemm}

\begin{proof}
At first, let us check that $\varphi=\varphi(M)\in \aut(G)$ for every $M=\left(\begin{smallmatrix}
\alpha & \beta \\
\varepsilon \beta & \alpha \\
\end{smallmatrix}\right)\in \mathcal{M}$. A straightforward computation implies that
$$\varphi\left(\left(\begin{smallmatrix}
1 & x & z\\
0 & 1 & y\\
0 & 0 & 1
\end{smallmatrix}\right)
\left(\begin{smallmatrix}
1 & a & c\\
0 & 1 & b\\
0 & 0 & 1
\end{smallmatrix}\right)\right)=\left(\begin{smallmatrix}
1 & \alpha(x+a)+\varepsilon\beta(y+b) & F_{\alpha,\beta}(x+a,y+b,xb+z+c)\\
0 & 1 & \beta(x+a)+\alpha(y+b)\\
0 & 0 & 1
\end{smallmatrix}\right)=\varphi\left(\left(\begin{smallmatrix}
1 & x & z\\
0 & 1 & y\\
0 & 0 & 1
\end{smallmatrix}\right)\right)\varphi\left(\left(\begin{smallmatrix}
1 & a & c\\
0 & 1 & b\\
0 & 0 & 1
\end{smallmatrix}\right)\right).$$
So $\varphi$ is an endomorphism of $G$. Suppose that $\varphi\left(\left(\begin{smallmatrix}
1 & x & z\\
0 & 1 & y\\
0 & 0 & 1
\end{smallmatrix}\right)\right)=I$, where $I$ is the identity matrix. Then $\alpha x+\varepsilon \beta y=\beta x+ \alpha y=F_{\alpha,\beta}(x,y,z)=0$ by the definition of $\varphi$. Since $\varepsilon$ is a nonsquare, $\det(M)=\alpha^2-\varepsilon \beta^2\neq 0$ and hence $x=y=0$. Together with $F_{\alpha,\beta}(x,y,z)=0$, this implies that $z=0$. Therefore, $\varphi$ has a trivial kernel. Thus, $\varphi\in \aut(G)$.

One can verify in a straightforward way that $\varphi(M_1M_2)=\varphi(M_1)\varphi(M_2)$. If $\varphi(M)=\id_G$ for some $M=\left(\begin{smallmatrix}
\alpha & \beta \\
\varepsilon \beta & \alpha \\
\end{smallmatrix}\right)\in \mathcal{M}$, then 

$$\varphi(M)\left(\left(\begin{smallmatrix}
1 & 1 & 0\\
0 & 1 & 0\\
0 & 0 & 1
\end{smallmatrix}\right)\right)=
\left(\begin{smallmatrix}
1 & 1 & 0\\
0 & 1 & 0\\
0 & 0 & 1
\end{smallmatrix}\right)\Rightarrow 
\left(\begin{smallmatrix}
1 & \alpha & \frac{\alpha\beta}{2}\\
0 & 1 & \beta\\
0 & 0 & 1
\end{smallmatrix}\right)=\left(\begin{smallmatrix}
1 & 1 & 0\\
0 & 1 & 0\\
0 & 0 & 1
\end{smallmatrix}\right) \Rightarrow M = \left(\begin{smallmatrix}
1 & 0 \\
0 & 1 \\
\end{smallmatrix}\right)
.$$

%So $\alpha=1$, $\beta=0$, and hence $M=I$. 
Thus, the mapping $M\mapsto \varphi(M)$ is a monomorphism.
\end{proof}

Let $K=\{\varphi(M):~M\in \mathcal{M}\}$. Due to Lemma~\ref{autheis}, $K$ is a cyclic subgroup of $\aut(G)$ of order~$q^2-1$. The next lemma collects properties of $K$-orbits.

\begin{lemm}\label{orbits}
The following statements hold:
\begin{enumerate}
\tm{1} $Z^\#\in \orb(K,G)$;

\tm{2} for every $X\in \orb(K,G\setminus Z)$, the set $X\cup \{e\}$ is a transversal for $Z$ in $G$; in particular, $|X|=q^2-1$.
\end{enumerate}
\end{lemm}

\begin{proof}
(1) \ Let $g=\left(\begin{smallmatrix}
1 & 0 & z\\
0 & 1 & 0\\
0 & 0 & 1
\end{smallmatrix}\right)\in Z$. Then 
$$\varphi(M)(g)=\left(\begin{smallmatrix}
1 & 0 & (\alpha^2-\varepsilon \beta^2)z\\
0 & 1 & 0\\
0 & 0 & 1
\end{smallmatrix}\right)\in Z$$ 
for every $M=\left(\begin{smallmatrix}
\alpha & \beta \\
\varepsilon \beta & \alpha \\
\end{smallmatrix}\right)\in \mathcal{M}$. So $Z^\#$ is a $K$-invariant set. Since $\varepsilon$ is nonsquare, 
$$\{\alpha^2-\varepsilon \beta^2:~\alpha,\beta\in \mathbb{F}_q,~(\alpha,\beta)\neq (0,0) \}=\mathbb{F}_q\setminus\{0\}.$$ 
Therefore $K$ is transitive on $Z^\#$ and hence $Z^\#\in \orb(K,G)$. Thus, part~$(1)$ of the lemma holds.

(2) \ Assume that $\varphi(M)(g)=g^\prime$ for some $g=\left(\begin{smallmatrix}
1 & x & z\\
0 & 1 & y\\
0 & 0 & 1
\end{smallmatrix}\right)\in G$, $g^\prime=\left(\begin{smallmatrix}
1 & x & z^\prime\\
0 & 1 & y\\
0 & 0 & 1
\end{smallmatrix}\right)\in Zg$, and $M=\left(\begin{smallmatrix}
\alpha & \beta \\
\varepsilon \beta & \alpha \\
\end{smallmatrix}\right)\in \mathcal{M}$. Then 
\begin{equation*}
 \begin{cases}
\alpha x+\varepsilon \beta y=x,
\\
\beta x+\alpha y=y.
\end{cases}\iff 
\begin{cases}
(\alpha -1) x+\varepsilon \beta y=0,
\\
\beta x+(\alpha-1) y=0.
\end{cases}
\end{equation*}
Since $(x,y)\neq (0,0)$, the determinant  of the above system is zero, i.e. $(\alpha-1)^2-\varepsilon \beta^2=0$. Since $\varepsilon$ is a nonsquare, the only solution of the latter equation is $\alpha=1,\beta=0$. i.e. $M=I$. The above discussion implies that: $(1)$ $K$ acts semiregularly on $G\setminus Z$ or, equivalently, each orbit of $K$ outside $Z$ is of size $q^2-1$; $(2)$ any two elements of $G\setminus Z$ from the same orbit of $K$ belong to distinct $Z$-cosets. Thus, part~$(2)$ of the lemma holds.  
\end{proof}

Put $\mathcal{A}=\cyc(K,G)$. Due to Lemma~\ref{orbits}, $Z^\#$ is a basic set of $\mathcal{A}$ and there are $\frac{q^3-q}{q^2-1}=q$ basic sets inside $G\setminus Z$. Clearly, $Z$ is an $\mathcal{A}$-subgroup. Given $i\in \mathbb{F}_q$, denote the basic set of $\mathcal{A}$ containing the element $\left(\begin{smallmatrix}
1 & 1 & i\\
0 & 1 & 0\\
0 & 0 & 1
\end{smallmatrix}\right)$ by $Y_i$. By part~$(2)$ of Lemma~\ref{orbits}, all $Y_i$ are pairwise distinct and of size~$q^2-1$. By the definition of $\varphi(M)$, we have
$$Y_i=\left\{\left(\begin{smallmatrix}
1 & \alpha & \gamma_i(\alpha,\beta)\\
0 & 1 & \beta\\
0 & 0 & 1
\end{smallmatrix}\right),~\alpha,\beta\in \mathbb{F}_q,~(\alpha,\beta)\neq (0,0)\right\},$$ 
where $\gamma_i(\alpha,\beta)=\frac{\alpha\beta}{2}+(\alpha^2-\varepsilon \beta^2)i$. It is easy to verify that
$$\left(\begin{smallmatrix}
1 & \alpha & \gamma_i(\alpha,\beta)\\
0 & 1 & \beta\\
0 & 0 & 1
\end{smallmatrix}\right)^{-1}=\left(\begin{smallmatrix}
1 & -\alpha & \gamma_{-i}(-\alpha,-\beta)\\
0 & 1 & -\beta\\
0 & 0 & 1
\end{smallmatrix}\right).$$
The latter equality implies that $Y_i^{(-1)}=Y_{-i}$. In particular, $Y_0^{(-1)}=Y_0$. 

Given $i\in \mathbb{F}_q$, put $X_i=Y_i\cup \{e\}$, where $e=I$ is the identity of $G$. Clearly, $X_i^{(-1)}=X_{-i}$ and $|X_i|=q^2$ for every~$i$.

\begin{prop}\label{multiplication}
In the above notations, 
\begin{equation}\label{link}
\underline{X_i}\cdot \underline{X_j}=
\begin{cases}
q^2 e+q(\underline{G\setminus Z}),~j=-i,\\
X_{\psi(i,j)}+(q+1)(\underline{G\setminus X_{\psi(i,j)}}),~j\neq-i,
\end{cases}
\end{equation}
where $\psi(i,j)=\frac{ij+\delta}{i+j}$ for some nonsquare $\delta\in \mathbb{F}_q$.
\end{prop}

\begin{proof}
Firstly, let us calculate the structure constants $c_{Y_iY_j}^T$ of $\mathcal{A}$, where $i,j\in \mathbb{F}_q$ and $T\in \mathcal{S}(\mathcal{A})$. Clearly, $c_{Y_iY_j}^{\{e\}}=q^2-1$ if $j=-i$ and $c_{Y_iY_j}^{\{e\}}=0$ otherwise. Each element from $Y_i$ is of the form $\left(\begin{smallmatrix}
1 & \alpha_1 & \gamma_i(\alpha_1,\beta_1)\\
0 & 1 & \beta_1\\
0 & 0 & 1
\end{smallmatrix}\right)$ and each element from $Y_j$ is of the form $\left(\begin{smallmatrix}
1 & \alpha_2 & \gamma_j(\alpha_2,\beta_2)\\
0 & 1 & \beta_2\\
0 & 0 & 1
\end{smallmatrix}\right)$. So each element entering the product $Y_iY_j$ is of the form $$g=\left(\begin{smallmatrix}
1 & \alpha_1+\alpha_2 & \gamma_i(\alpha_1,\beta_1)+\gamma_j(\alpha_2,\beta_2)+\alpha_1\beta_2\\
0 & 1 & \beta_1+\beta_2\\
0 & 0 & 1
\end{smallmatrix}\right).$$
The structure constant $c_{Y_iY_j}^{Z^\#}$, where $i,j\in \mathbb{F}_q$, is equal to the number of solutions $\alpha_1,\beta_1,\alpha_2,\beta_2$ of the matrix equation $g=\left(\begin{smallmatrix}
1 & 0 & 1\\
0 & 1 & 0\\
0 & 0 & 1
\end{smallmatrix}\right)$ 
 such that 
\begin{equation}\label{zeros}
(\alpha_1,\beta_1)\neq (0,0)~\text{and}~(\alpha_2,\beta_2)\neq (0,0).
\end{equation}
The latter matrix equation is equivalent to the system of equations
\begin{equation}\label{system1}
\begin{cases}
\alpha_2=-\alpha_1,\\
\beta_2=-\beta_1,\\
(\alpha_1^2-\varepsilon \beta_1^2)(i+j)=1.
\end{cases}
\end{equation}
If $i+j=0$, then the third equation in System~\eqref{system1} does not have a solution and hence
\begin{equation}\label{const1}
c_{Y_iY_{-i}}^{Z^\#}=0.
\end{equation}
If $i+j\neq 0$, then the third equation in System~\eqref{system1} is a Pell equation with the nonsquare coefficient $\varepsilon$. Such equation has $q+1$ solutions (see, e.g.,~\cite[Theorem~5.3]{Cohen}). Clearly,~\eqref{zeros} holds for all of these solutions. Therefore
\begin{equation}\label{const2} 
c_{Y_iY_j}^{Z^\#}=q+1,~i+j\neq 0.
\end{equation}

Now given $i,j,k\in \mathbb{F}_q$, let us calculate the structure constant $c_{Y_iY_j}^{Y_k}$. It is equal to the number of solutions $\alpha_1,\beta_1,\alpha_2,\beta_2$ of the matrix equation $g=\left(\begin{smallmatrix}
1 & 1 & k\\
0 & 1 & 0\\
0 & 0 & 1
\end{smallmatrix}\right)$ 
satisfying~\eqref{zeros}. The above matrix equation can be written as the following system of equations:
\begin{equation}\label{system2}
\begin{cases}
\alpha_2=1-\alpha_1,\\
\beta_2=-\beta_1,\\
\frac{\alpha_1\beta_1+\alpha_2\beta_2}{2}+\alpha_1\beta_2+(\alpha_1^2-\varepsilon \beta_1^2)i+(\alpha_2^2-\varepsilon \beta_2^2)j=k.
\end{cases}
\end{equation}
The third equation in System~\eqref{system2} can be reduced using the first and the second equations to the following one
\begin{equation}\label{pell}
\alpha_1^2(i+j)-2\alpha_1j-\frac{\beta_1}{2}-\varepsilon \beta_1^2(i+j)+j-k=0.
\end{equation}
 
Suppose first that $i+j=0$. Then~\eqref{pell} is reduced to
\begin{equation}\label{linear}
\beta_1=2(j-k)-4\alpha_1j.
\end{equation} 
The latter equation and hence System~\eqref{system2} has $q$ solutions. It remains to verify how many of these solutions satisfy~\eqref{zeros} the second part of which is equivalent to $(\alpha_1,\beta_1)\neq (1,0)$ in our case. One can see that $(0,0)$ ($(1,0)$, respectively) is a solution of~\eqref{linear} if and only if $k=j$ ($k=-j=i$, respectively). Note that both $(0,0)$ and $(1,0)$ are solutions of~\eqref{linear} if and only if $k=j=i=0$. Thus,
\begin{equation}\label{const3}
c_{Y_iY_{-i}}^{Y_k}=
\begin{cases}
q,~k\notin\{i,-i\},\\
q-1,~k\in\{i,-i\}~\text{and}~i\neq 0,\\
q-2,~k=i=-i=0.
\end{cases}
\end{equation}

Now let $i+j\neq 0$. Then~\eqref{pell} is equivalent to the next one:
\begin{equation}\label{pell2}
\left(\alpha_1-\frac{j}{i+j}\right)^2-\varepsilon \left(\beta_1+\frac{1}{4(i+j)\varepsilon}\right)^2=\frac{j^2}{(i+j)^2}-\frac{1}{16(i+j)^2\varepsilon}+\frac{k-j}{i+j}.
\end{equation}
Observe that~\eqref{pell2} is a Pell equation with the nonsquare coefficient $\varepsilon$. If the right-hand side of~\eqref{pell2} is equal to~$0$, then this equation has a unique solution $(\alpha_1,\beta_1)=(j(i+j)^{-1},(4(i+j)\varepsilon)^{-1})$ and hence System~\eqref{system2} has a unique solution satisfying~\eqref{zeros}. One can compute that the right-hand side of~\eqref{pell2} is equal to~$0$ if and only if
$$k=\psi(i,j)=\frac{ij+\delta}{i+j},$$
where $\delta=\frac{\varepsilon}{16}$ is a nonsquare. Therefore
\begin{equation}\label{const4}
c_{Y_iY_j}^{Y_{\psi(i,j)}}=1,~i+j\neq 0.
\end{equation}

If the right-hand side of~\eqref{pell2} is not equal to~$0$, then~\eqref{pell2} and hence System~\eqref{system2} has $q+1$ solutions. Let us check how many of these solutions satisfy~\eqref{zeros}. Recall that~\eqref{zeros} is equivalent in our case to $(\alpha_1,\beta_1)\notin \{(0,0),(1,0)\}$. One can verify that $(0,0)$ ($(1,0)$, respectively) is a solution of~\eqref{pell2} if and only if $k=j$ ($k=i$, respectively). Therefore
\begin{equation}\label{const5}
c_{Y_iY_j}^{Y_k}=
\begin{cases}
q+1,~i+j\neq 0~\text{and}~k\notin\{i,j,\psi(i,j)\},\\
q,~i+j\neq 0,~i\neq j,~\text{and}~k\in\{i,j\},\\
q-1,~i+j\neq 0,~k=i=j.
\end{cases}
\end{equation}

Clearly, 
$$\underline{X_i}\cdot\underline{X_j}=(\underline{Y_i}+e)(\underline{Y_j}+e)=(c_{Y_iY_j}^{\{e\}}+1)e+c_{Y_iY_j}^{Z^\#}\underline{Z}^\#+(c_{Y_iY_j}^{Y_i}+1)\underline{Y_i}+(c_{Y_iY_j}^{Y_j}+1)\underline{Y_j}+\sum \limits_{k\in \mathbb{F}_q\setminus\{i,j\}} c_{Y_iY_j}^{Y_k} \underline{Y_k}$$
for all $i,j\in \mathbb{F}_q$. Now~\eqref{link} follows from the above equality,~\eqref{const1},~\eqref{const2},~\eqref{const3},~\eqref{const4}, and~\eqref{const5}.
\end{proof}

The next corollary immediately follows from Proposition~\ref{multiplication}. 

\begin{corl}\label{heissystem}
In the above notations, $\mathcal{L}=\{X_i:~i\in \mathbb{F}_q\}$ is a closed linked system of semiregular RDSs in $G$ with forbidden subgroup $Z$, parameters $(q^2,q,q^2,q,q,1,q+1)$, and pair of characteristic functions $(\chi,\psi)$, where $\chi(i)=-i$ and $\psi(i,j)=\frac{ij+\delta}{i+j}$.
\end{corl}

Note that the set $X_0$ is a special case of a reversible RDS from~\cite[Theorem~1.3]{CHL}. Since $X_0$ is a reversible RDS with forbidden subgroup $Z$ and parameters $(q^2,q,q^2,q)$, Lemma~\ref{rdspds} implies that $X_0^\#\cup Z^\#$ is a reversible PDS with parameters~$(q^3,q^2+q-2,q-2,q+2)$. So we obtain the next corollary.

\begin{corl}\label{pdsheis}
Let $q$ be an odd prime power. There is a reversible PDS with parameters~$(q^3,q^2+q-2,q-2,q+2)$ in a Heisenberg group of dimension~$3$ over $\mathbb{F}_q$. 
\end{corl}

Since a Heisenberg group is a central product of Heisenberg groups of dimension~$3$, the corollary below which is a refinement of Theorem~\ref{main2} immediately follows from Proposition~\ref{product} and Corollary~\ref{heissystem}. Observe that parameters $\mu$ and $\nu$ can be computed recurrently by formulas from Proposition~\ref{product} or by~\eqref{munu-reg} from $m$, $n$, and $k$.

\begin{corl}\label{heis2r}
Let $q$ be an odd prime power, $r\geq 1$, and $G$ a Heisenberg group of dimension~$2r+1$ over $\mathbb{F}_q$. There is a closed linked system $\mathcal{L}=\{X_i:~i\in \mathbb{F}_q\}$ of semiregular RDSs in $G$ with forbidden subgroup $Z(G)$, parameters~
$$(q^{2r},q,q^{2r},q^{2r-1},q,q^{2r-1}-q^r+q^{r-1},q^{2r-1}+q^{r-1}),$$
and pair of characteristic functions $(\chi,\psi)$, where $\chi(i)=-i$ and $\psi(i,j)=\frac{ij+\delta}{i+j}$ for a nonsquare $\delta\in \mathbb{F}_q$. 
\end{corl}

According to Proposition~\ref{function}, one can construct the group $\mathbb{F}_q^\infty$ associated with $\mathcal{L}$ on the set $\mathbb{F}_q\cup \{\infty\}$, where $\infty$ is the identity element and the unary operation $\widehat{\chi}$ and the binary operation $\widehat{\psi}$ are constructed from $\chi(i)=-i$ and $\psi(i,j)=\frac{ij+\delta}{i+j}$, respectively. Observe that the set $\mathbb{F}_q^\infty$ can be considered also as a projective line $\PL(q)$ (see~\cite[p.~50]{Wilson}).

\begin{lemm}\label{cyclgroup}
The group $\mathbb{F}_q^\infty=\mathbb{F}_q\cup \{\infty\}$ associated with $\mathcal{L}$ is a cyclic group of order~$q+1$. 
\end{lemm}

\begin{proof}
Recall that the set 
$$\mathcal{M}(\delta)=\{\left(\begin{smallmatrix}
\alpha & \beta \\
\delta\beta & \alpha \\
\end{smallmatrix}\right):~\alpha,\beta\in \mathbb{F}_q,~(\alpha,\beta)\neq (0,0)\}$$ 
with standard matrix multiplication is isomorphic to the multiplicative group $\mathbb{F}_{q^2}^\times\cong C_{q^2-1}$. Put 
$$\mathcal{M}_0=\{\left(\begin{smallmatrix}
\alpha & 0 \\
0 & \alpha \\
\end{smallmatrix}\right):~\alpha\in \mathbb{F}_q\setminus\{0\}\}.$$
Clearly, $\mathcal{M}_0$ is a subgroup of $\mathcal{M}(\delta)$ of order~$q-1$ and hence $\mathcal{M}(\delta)/\mathcal{M}_0 \cong C_{q+1}$. To prove the lemma, it is enough to verify that $\mathbb{F}_q^\infty$ is isomorphic to $\mathcal{M}(\delta)/\mathcal{M}_0$ (with the standard multiplication induced from the matrix multiplication in $\mathcal{M}$). Note that the set 
$$\mathcal{T}=\{\left(\begin{smallmatrix}
\alpha & 1 \\
\delta & \alpha \\
\end{smallmatrix}\right):~\alpha\in \mathbb{F}_q\}\cup \{I\}$$
is a transversal for $\mathcal{M}_0$ in $\mathcal{M}(\delta)$. Let $f:\mathbb{F}_q^\infty\rightarrow \mathcal{M}(\delta)/\mathcal{M}_0$ such that
$$f(i)=\mathcal{M}_0\left(\begin{smallmatrix}
i & 1 \\
\delta & i \\
\end{smallmatrix}\right),~i\in \mathbb{F}_q,~\text{and}~f(\infty)=\mathcal{M}_0.$$
Since $\mathcal{T}$ is a transversal for $\mathcal{M}_0$ in $\mathcal{M}(\delta)$, the mapping $f$ is a bijection. Given $i,j \in \mathbb{F}_q$ such that $i+j\neq 0$, we have
$$f(\widehat{\psi}(i,j))=\mathcal{M}_0\left(\begin{smallmatrix}
\frac{ij+\delta}{i+j} & 1 \\
\delta & \frac{ij+\delta}{i+j} \\
\end{smallmatrix}\right)=\mathcal{M}_0\left(\begin{smallmatrix}
ij+\delta & i+j \\
\delta(i+j) & ij+\delta \\
\end{smallmatrix}\right)=(\mathcal{M}_0\left(\begin{smallmatrix}
i & 1 \\
\delta & i \\
\end{smallmatrix}\right))(\mathcal{M}_0\left(\begin{smallmatrix}
j & 1 \\
\delta & j \\
\end{smallmatrix}\right))=f(i)f(j),$$
$$f(\widehat{\psi}(i,-i))=f(\infty)=\mathcal{M}_0=(\mathcal{M}_0\left(\begin{smallmatrix}
i & 1 \\
\delta & i \\
\end{smallmatrix}\right))(\mathcal{M}_0\left(\begin{smallmatrix}
-i & 1 \\
\delta & -i \\
\end{smallmatrix}\right))=f(i)f(-i),$$
$$f(\widehat{\psi}(\infty,\infty))=f(\infty)=\mathcal{M}_0=f(\infty)f(\infty),$$
$$f(\widehat{\psi}(i,\infty))=f(i)=(\mathcal{M}_0\left(\begin{smallmatrix}
i & 1 \\
\delta & i \\
\end{smallmatrix}\right))\mathcal{M}_0=f(i)f(\infty),$$
and
$$f(\widehat{\psi}(\infty,i))=f(i)=\mathcal{M}_0(\mathcal{M}_0\left(\begin{smallmatrix}
i & 1 \\
\delta & i \\
\end{smallmatrix}\right))=f(\infty)f(i).$$
Thus, $f$ is an isomorphism from $\mathbb{F}_q^\infty$ to $\mathcal{M}(\delta)/\mathcal{M}_0$ and hence the group $\mathbb{F}_q^\infty$ associated with $\mathcal{L}$ is isomorphic to $C_{q+1}$.
\end{proof}

\section{RDSs in an extraspecial $p$-group of exponent~$p^2$}

The main goal of this section is to prove Theorem~\ref{main1}. To do this, we construct semiregular RDSs in an extraspecial group of order~$p^3$ and exponent~$p^2$ as fusions of some $S$-ring in which every basic set outside the elementary abelian subgroup of order~$p^2$ corresponds to an RDS (see Proposition~\ref{p24}). In fact, a PDS constructed in~\cite{Swartz} is also a fusion of the above $S$-ring. We use this observation for required computations. In the end of the section, we consider the case of $2$-groups.

\subsection{Case of odd~$p$.}\label{p-odd}

Let $p$ be an odd prime and $G\cong M_{p^3}$ an extraspecial group of order~$p^3$ and exponent $p^2$. Then $G$ can be presented as a semidirect product $G=\langle x \rangle \rtimes \langle y \rangle$, where $|x|=p^2$, $|y|=p$, and $yxy^{-1}=xz$ for $z=x^p$. Put $Y=\langle y \rangle$ and $Z=\langle z \rangle$. Clearly, $Z(G)=Z$ and $Y\times Z\cong C_p\times C_p$. Note that each element $g$ of $G$ can be presented in a unique way in the form $g=x^\alpha y^\beta z^\gamma$, where $\alpha\in \{0,\xi^0,\ldots,\xi^{p-2}\}$, $\xi$ is the $p$th power of a primitive root modulo $p^2$, and $\beta,\gamma\in \{0,\ldots,p-1\}$.

There are $\varphi,\psi\in \aut(G)$ such that 
$$\sigma:(x,y,z)\mapsto (xyz^{\frac{p+1}{2}},y,z)~\text{and}~\tau:(x,y,z)\mapsto (x^{\xi},y,z^{\eta(\xi)}),$$
where $\eta=\eta(\xi)\in \mathbb{Z}_p$ is such that $\eta \equiv\xi \mod~p$ (see~\cite{Winter}). One can check in a straightforward way that $|\sigma|=p$, $|\tau|=p-1$, and $\sigma^\tau=\sigma^{\eta}$. Therefore the group $K=\langle \sigma \rangle \rtimes \langle \tau \rangle$ is a Frobenius group of order~$p(p-1)$ (see also~\cite{Swartz}). In fact, $K\cong \out(G)$ due to~\cite{Winter}. Put $\mathcal{A}=\cyc(K,G)$.

\begin{lemm}\label{p21}
The set $\mathcal{S}(\mathcal{A})$ consists of the following sets:
$$\{y^i\},~Z^\#y^i,~X_i=X_0y^i,~i\in\{0,\ldots,p-1\},$$
where
$$X_0=\{x^\alpha y^\beta z^{\frac{\eta(\alpha)\beta}{2}}:~\alpha\in \{\xi^0,\ldots,\xi^{p-2}\},~\beta\in\{0,\ldots,p-1\}\}.$$
\end{lemm}

\begin{proof}
Since $y$ is fixed by $\sigma$ and $\tau$, each element of $K$ fixes $y$ and hence $y^i$ for every $i\in\{0,\ldots,p-1\}$. So $\{y^i\}\in \mathcal{S}(\mathcal{A})$. Observe that $\langle \tau \rangle$ acts transitively on $Z^\#$ because $\eta=\eta(\xi)$ is a primitive root modulo~$p$. Therefore $Z^\#\in \mathcal{S}(\mathcal{A})$. The set $X_0$ is an orbit of $K$ and hence a basic set of $\mathcal{A}$ by~\cite[Lemma~1(v)]{Swartz} (this also can be verified explicitly using the definitions of $G$ and $K$). Since $X_0,Z^\#,\{y^i\}\in \mathcal{S}(\mathcal{A})$, we obtain $X_i=X_0y^i,Z^\#y^i\in \mathcal{S}(\mathcal{A})$. Clearly, all the sets $Z^\#y^i$ are pairwise disjoint. One can see that $Y\nleq \rad(X_0)$ and hence all the sets $X_i$ are also pairwise disjoint. This implies that the sets $\{y^i\}$, $Z^\#y^i$, $X_i$ form a partition of $G$ and we are done.
\end{proof}

\begin{lemm}\label{p22}
For every $i\in \{0,\ldots,p-1\}$, the following statements hold:
\begin{enumerate}

\tm{1} $X_i=X_i^{(-1)}$;

\tm{2} $Y_i=X_i\cup Y$ and $Z_i=X_i\cup Z$ are (both left and right) transversals in $G$ for $Z$ and $Y$, respectively.
\end{enumerate}
\end{lemm}

\begin{proof}
Note that $\tau_0=\tau^{\frac{p-1}{2}}\in K$ acts trivially on $Y$ and inverses every nontrivial element of $\langle x \rangle\cong C_{p^2}$. Each $X_i$ contains an element $x_i$ from $\langle x \rangle$. So $x_i^{-1}=\tau_0(x_i)\in \tau_0(X_i)=X_i$, where the latter equality holds because $X_i$ is an orbit of $K$ and $\tau_0\in K$. This yields that $X_i\cap X_i^{(-1)}\neq \varnothing$. Since $X_i$ is a basic set of $\mathcal{A}$, the latter implies that $X_i=X_i^{(-1)}$  and part~$(1)$ holds. Part~$(2)$ follows from the definition of $X_i$.
\end{proof}

\begin{lemm}\label{sigma}
For every $i\in \{0,\ldots,p-1\}$, there is $\sigma_i\in \aut(G)$ such that $\sigma_i(X_0)=X_i$, $\sigma_i(Y)=Y$, and $\sigma_i(Z)=Z$.
\end{lemm}

\begin{proof}
Due to~\cite{Winter}, there exists $\sigma_i\in \aut(G)$ such that $\sigma_i:(x,y,z)\mapsto (xy^i,y,z)$ for every $i\in\{0,\ldots,p-1\}$. By the definition of $\sigma_i$, we have $\sigma_i(Y)=Y$ and $\sigma_i(Z)=Z$. Given $\alpha\in \{\xi^0,\ldots,\xi^{p-2}\}$ and $\beta\in\{0,\ldots,p-1\}$, a straightforward computation shows that
$$\sigma_i(x^\alpha y^\beta z^{\frac{\eta(\alpha)\beta}{2}})=(xy^i)^\alpha y^\beta z^{\frac{\eta(\alpha)\beta}{2}}=x^\alpha y^{\beta+i\eta(\alpha)}z^{\frac{\eta(\alpha)\beta}{2}+\frac{i\eta(\alpha)(\eta(\alpha)-1)}{2}}=$$
$$=(x^\alpha y^{\beta+i(\eta(\alpha)-1)}z^{\frac{\eta(\alpha)(\beta+i(\eta(\alpha)-1)}{2}})y^i\in X_0y^i=X_i,$$
where the second equality holds by successive applying of $yxy^{-1}=xz$. Therefore
$$\sigma_i(X_0)=X_i$$ 
and we are done. 
\end{proof}

\begin{lemm}\label{p23}
For every $i\in\{0,\ldots,p-1\}$, the set $S_i=X_i\cup Y^\# \cup Z^\#$ is a reversible PDS with parameters $(p^3,p^2+p-2,p-2,p+2)$.
\end{lemm}

\begin{proof}
The set $S_0$ is exactly a reversible PDS with parameters $(p^3,p^2+p-2,p-2,p+2)$ constructed in~\cite{Swartz}. From Lemma~\ref{sigma} it follows that $\sigma_i(S_0)=S_i$. Therefore $S_i$ is a reversible PDS with the same parameters as $S_0$.
\end{proof}

\begin{prop}\label{p24}
For every $i\in\{0,\ldots,p-1\}$, the sets $Y_i=X_i\cup Y$ and $Z_i=X_i\cup Z$ are semiregular reversible RDSs with parameters $(p^2,p,p^2,p)$ and forbidden subgroups $Z$ and $Y$, respectively. 
\end{prop}

\begin{proof}
It is easy to see that $S_i=Y_i^\#\cup Z^\#=Z_i^\#\cup Y^\#$. Therefore
\begin{equation}\label{s1}
(\underline{S_i}+2e)^2=(\underline{Y_i}+\underline{Z})^2=\underline{Y_i}^2+(\underline{Y_i}\cdot\underline{Z}+\underline{Z}\cdot\underline{Y_i})+p\underline{Z}=\underline{Y_i}^2+2\underline{G}+p\underline{Z}
\end{equation}
and
\begin{equation}\label{s2}
(\underline{S_i}+2e)^2=(\underline{Z_i}+\underline{Y})^2=\underline{Z_i}^2+(\underline{Z_i}\cdot\underline{Y}+\underline{Y}\cdot \underline{Z_i})+p\underline{Y}=\underline{Z_i}^2+2\underline{G}+p\underline{Y}.
\end{equation}
In the above computations, we used the equalities $Y_iZ=ZY_i=Z_iY=YZ_i=G$ which hold by Lemma~\ref{p22}(2). On the other hand,
\begin{equation}\label{s3}
(\underline{S_i}+2e)^2=\underline{S_i}^2+4\underline{S_i}+4e=(p^2+p+2)e+(p+2)\underline{G}^\#
\end{equation}
by Lemma~\ref{p23}. Equations~\eqref{s1} and~\eqref{s3} imply that
$$\underline{Y_i}^2=p^2e+p(\underline{G}-\underline{Z}),$$
whereas the equations~\eqref{s2} and~\eqref{s3} imply that
$$\underline{Z_i}^2=p^2e+p(\underline{G}-\underline{Y}).$$
Thus, $Y_i$ and $Z_i$ are RDSs with forbidden subgroups~$Z$ and $Y$, respectively, with parameters $(p^2,p,p^2,p)$. Since $X_i=X_i^{(-1)}$ (Lemma~\ref{p22}(1)), $Y_i$ and $Z_i$ are reversible.
\end{proof}

It should be mentioned that the sets $Y_i$ and the sets $Z_i$ do not form a linked system. Observe also that the sets $Z_i$ provide examples of RDSs with a nonnormal forbidden subgroup.

The corollary below immediately follows from Proposition~\ref{p24}.

\begin{corl}\label{expp2}
Let $p$ be an odd prime and $G$ an extraspecial group of order~$p^3$ and exponent~$p^2$. There is a semiregular reversible RDS in $G$ with forbidden subgroup~$Z(G)$ and parameters~$(p^{2},p,p^{2},p)$.
\end{corl}

Every extraspecial $p$-group is a central product of extraspecial groups of order~$p^3$ and there are two extraspecial groups of order~$p^3$, namely, a Heisenberg group of order~$p^3$ and $M_{p^3}$. Therefore Theorem~\ref{main1} follows from Proposition~\ref{product0}, Corollary~\ref{heis2r}, and Corollary~\ref{expp2}.

\subsection{Case~$p=2$.}

There are two extraspecial groups of order~$8$, namely, the dihedral group $D_8$ and the quaternion group $Q_8$. The first of them is isomorphic to a Heisenberg group of dimension~$3$ over $\mathbb{F}_2$. From~\cite[Theorem~3.2,~Corollary~4.3]{EH} it follows that: (1) there is no a relative to the center difference set in $D_8$; (2) there is a unique (up to a group automorphism) a relative to a nonnormal subgroup of order~$2$ difference set in $D_8$; (3) there is a relative to the center difference set in $Q_8$. Based on the latter fact, we prove the following statement.

\begin{prop}\label{2group}
Let $G$ be a central product of~$r\geq 1$ quaternion groups $Q_8$. There is a closed linked system of semiregular RDSs in~$G$ with forbidden subgroup $Z(G)$ and parameters 
$$(2^{2r},2,2^{2r},2^{2r-1},2,2^{2r-1}-2^r+2^{r-1},2^{2r-1}+2^{r-1}).$$
\end{prop}

\begin{proof}
Firstly, let $r=1$ and $G=\langle a,b:~a^4=e,~a^2=b^2,~bab^{-1}=a^{-1}\rangle \cong Q_8$. Put $X_1=\{e,a,b,ba\}$ and $X_2=X_1^{(-1)}$. Due to~\cite[Corollary~3.3]{EH}, $X_1$ and hence $X_2$ are semiregular RDSs in $G$ with forbidden subgroup $Z=Z(G)=\langle a^2 \rangle \cong C_2$ and parameters~$(4,2,4,2)$. A straightforward computation implies that
$$\underline{X_1}^2=\underline{X_2}+3(\underline{G}-\underline{X_2})~\text{and}~\underline{X_2}^2=\underline{X_1}+3(\underline{G}-\underline{X_1}).$$
Therefore $\mathcal{L}=\{X_1,X_2\}$ is a closed linked system of semiregular RDSs in $G$ with forbidden subgroup~$Z$, parameters~$(4,2,4,2,2,1,3)$ and characteristic function $f_{\mathcal{L}}(i,i)=3-i$ for $i\in\{1,2\}$. For $n\geq 2$, the proposition follows from the above discussion,~\eqref{munu-reg}, and Proposition~\ref{product}.
\end{proof}

\section{Miscellaneous}

\subsection{Connection with known graphs.} 

One can check that the Cayley graphs whose connection sets are the reversible RDSs constructed in Sections~$6$ and~$7$ are isomorphic to the distance-regular antipodal graph of diameter~$3$ from~\cite[Proposition~12.5.1]{BCN} which is known also as the Thas-Somma graph (see~\cite{Somma,Thas}). Let us recall the construction of the above graph. Let $q$ be a prime power, $r\geq 1$, and $B$ a nondegenerate sympectic form on $\mathbb{F}_q^{2r}$. Let $\Gamma=\Gamma_{2r+1}(q)$ be the graph with vertex set $\mathbb{F}_q^{2r+1}$, where two vertices $(a,\alpha),(b,\beta)\in \mathbb{F}_q^{2r}\times \mathbb{F}_q$ are adjacent if and only if $B(a,b)=\alpha-\beta$ and $a\neq b$. It should be mentioned that $\Gamma$ does not depend on $B$, i.e. each two graphs constructed in such way from two distinct nondegenerate symplectic forms are isomorphic. The graph $\Gamma$ is a distance-regular antipodal graph of diameter~$3$ with intersection array $\{q^r-1,q^r-q^{r-1},1;1,q^{r-1},q^r-1\}$ (see~\cite[Section~12.4]{BCN}).

\subsection{Amorphic $S$-rings.} 

We do not know whether there exists an amorphic $S$-ring of Latin square type of rank at least~$4$ over a non-$p$-group which can be used for construction linked systems of RDSs from Section~$5$. It would be interesting to investigate this question and establish whether the above construction leads to linked systems of RDSs over non-$p$-groups. 

If a group $G$ has a  reversible Hadamard difference set (see~\cite[Section~VI.12]{BJL} for the definition), then there is an amorphic $S$-ring of Latin square type of rank~$3$ over $G$. According to Section~$5$, this $S$-ring can be used for construction of the trivial linked system consisting of one reversible RDS. It should be mentioned that in this case we obtain a construction of RDS from~\cite[Corollary~2.11]{AJP}. Note that due to~\cite[Theorem~4.2]{JS} each group of the form  
$$C_4^b \times C_{2^{c_1}}^2 \times \cdots C_{2^{c_j}}^2\times C_2^2 \times C_3^{2a} \times C_{p_1}^4 \times \cdots \times C_{p_i}^4,$$
where $p_1,\ldots,p_i$ are (not necessarily distinct) odd primes, and $a$, $b$, $c_1,\ldots,c_j$ are nonnegative integers, has a reversible Hadamard difference set. This leads us to reversible RDSs over non-$p$-groups.

\subsection{Equivalence and isomorphism}

Let $X\subseteq G$ be an RDS with a forbidden subgroup $N$. Then the set $\dev(X)=\{Xg:~g\in G\}$ is a block set of a symmetric block divisible design with point set $G$.

Let $G,\tilde{G}$ be finite groups. Two RDSs $X\subseteq G,\tilde{X}\subseteq\tilde{G}$  are called \emph{isomorphic} if the designs $(G,\dev(X))$ and $(\tilde{G},\dev(\tilde{X}))$ are isomorphic, i.e. there exists a bijection $\varphi:G\rightarrow \tilde{G}$ such that $\dev(\tilde{X}) = \{\varphi(S):~S\in\dev(X)\}$. If a design isomorphism may be realized by group isomorphism, the sets $X$ and $\tilde{X}$ are called \emph{equivalent}. For example, it follows from the Lemma~\ref{sigma} that the RDSs $Y_i$, $i=0,...,p-1$ are pairwise equivalent. The same is true for the sets $Z_i$, $i=0,...,p-1$. 

It is easy to see that equivalent RDSs are always isomorphic. On the other hand, there are numerous examples of isomorphic but inequivalent RDSs.

One can extend the above definitions to the systems of linked RDSs, but this is beyond the scope of this paper.
We hope to revisit this topic in a forthcoming paper.


\begin{thebibliography}{list}



\bibitem{AJP} 
\emph{K.~T.~Arasu,~D.~Jungnickel,~A.~Pott}, Divisible difference sets with multiplier~$-1$, J. Algebra, \textbf{133} (1990), 35--62.


\bibitem{BJL}
\emph{T.~Beth, D.~Jungnickel, H.~Lenz}, Design Theory, 2nd edition, Cambridge University Press, Cambridge (1999).





\bibitem{BCN}
\emph{A.~Brouwer, A.~Cohen, A.~Neumaier}, Distance-regular graphs, Springer, Heidelberg (1989).



\bibitem{CS}
\emph{P.~J.~Cameron}, On groups with several doubly-transitive permutation representations, Math. Z., \textbf{128} (1972), 1--14.


\bibitem{CHL} 
\emph{Y.~Q.~Chen,~K.~J.~Horadam,~W.~H.~Liu}, Relative difference sets fixed by inversion (iii) --- cocycle theoretical approach, Discrete Math., \textbf{308}, No.~13 (2008), 2764--2775.




\bibitem{CP}
\emph{G.~Chen, I.~Ponomarenko}, Coherent configurations, Central China Normal University Press, Wuhan (2019).



\bibitem{Cohen}
\emph{B.~Cohen}, Chebyshev polynomials and Pell equations over finite fields, Czechoslov. Math. J., \textbf{71} (2021), 491--510.



\bibitem{DMM}
\emph{E.~R. van Dam, W.~Martin, M.~Muzychuk}, Uniformity in association schemes and coherent configurations: cometric Q-antipodal schemes and linked systems, J. Comb. Theory, Ser. A, \textbf{120} (2013), 1401--1439.


\bibitem{DM}
\emph{E.~R. van Dam, M.~Muzychuk}, Some implications on amorphic association schemes, J. Comb. Theory, Ser. A, \textbf{117} (2010), 111--127.


\bibitem{Dav}
\emph{J.~A.~Davis}, A Note on products of relative difference sets, Des. Codes Cryptogr., \textbf{1} (1991), 117--119.


\bibitem{DMP}
\emph{J.~A.~Davis, W.~Martin, J.~Polhill}, Linking systems in nonelementary abelian groups, J. Comb. Theory, Ser. A, \textbf{123} (2014), 92--103.


\bibitem{DPS1}
\emph{J.~A.~Davis, J.~Polhill, K.~Smith}, Relative and almost linking systems, J. Algebraic Comb., \textbf{50}, No.~1 (2019), 113--118.



\bibitem{DPS2}
\emph{J.~A.~Davis, J.~Polhill, K.~Smith}, Relative linking systems of difference sets and linking systems of relative difference sets, in: Finite fields and their applications, Proc. the 14th Int. Conf. on Finite Fields and their Applications, ed. J.~A.~Davis, De Gruyter Proceedings in Mathematics (2019), 43--50.




\bibitem{EH}
\emph{D~T.~Elvira, Y.~Hiramine}, On non-abelian semi-regular relative difference sets, in: Finite Fields and Applications, eds. D.~Jungnickel et. al., Springer, Berlin (2001), 122--127.


\bibitem{Gur}
\emph{R.~Guralnick}, Invertible preservers and algebraic groups, Linear Algebra Appl., \textbf{212-213} (1994), 249--257.



\bibitem{GH}
\emph{C.~Godsil, A.~Hensel}, Distance regular covers of the complete graph, J. Combin. Theory Ser. B, \textbf{56} (1992), 205--238.



\bibitem{GIK}
\emph{J.~Gol'fand, A.~Ivanov, M.~Klin}, Amorphic cellular rings, in: Investigations in algebraic theory of combinatorial objects, eds. I.~Faradzev, et al., Kluwer, Dordrecht (1994), 167--186.

\bibitem{H95} 
\emph{D.~G.~Higman}, Rank 5 association schemes and triality, Linear Algebra Appl., \textbf{226--228} (1995), 197--222.



\bibitem{Hir} 
\emph{Y.~Hiramine}, On non-symmetric relative difference sets, Hokkaido Math. J., \textbf{37}, No.~1 (2008), 427--435.





\bibitem{JLS}
\emph{J.~Jedwab, S.~Li, S.~Simon}, Linking systems of difference sets, J. Comb. Des., \textbf{27}, No.~3 (2019), 161--187.



\bibitem{JS}
\emph{D.~Jungnickel, B.~Schmidt}, Difference sets: an update, in: Geometry, combinatorial designs and related structures, Proc. First Pythagorean Conf., eds. J.W.P. Hirschfeld et. al., Cambridge University Press (1997), 89--112.





\bibitem{KhS}
\emph{H.~Kharaghani,~S.~Suda}, Linked systems of symmetric group divisible designs, J. Algebr. Comb., \textbf{47} (2018), 319--343.





\bibitem{LM} 
\emph{K.~H.~Leung,~S.~L.~Ma}, Constructions of partial difference sets and relative difference sets on $p$-groups, Bull. London Math. Soc., \textbf{22} (1990), 533--539.




\bibitem{LN}
\emph{R.~Lidl,~H.~Niederreiter}, Finite fields, Cambridge University Press, Second edition (1997).




\bibitem{Ma}
\emph{S.~L.~Ma}, A survey of partial difference sets, Des. Codes Cryptogr., \textbf{4} (1994), 221--261.


\bibitem{MS}
\emph{S.~L.~Ma, B.~Schmidt}, Relative $(p^a,p^b,p^a,p^{a-b})$-difference sets: a unified exponent bound and a local ring construction, Finite Fields Appl., \textbf{6} (2000), 1--22.






\bibitem{Pott2}
\emph{A.~Pott}, A survey on relative difference set, in: Groups, Difference Sets, and the Monster, eds. K. T. Arasu et al., De Gruyter (1996), 195--232.




\bibitem{Ry} 
\emph{G.~Ryabov}, On separability of Schur rings over abelian $p$-groups, Algebra Log., \textbf{57}, No.~1 (2018), 49--68.




\bibitem{Schmidt}
\emph{B.~Schmidt}, On $(p^a,p^b,p^a,p^{a-b})$-relative difference sets, J. Alg. Comb., \textbf{6} (1997), 279--297.



\bibitem{Swartz}
\emph{E.~Swartz}, A construction of a partial difference set in the extraspecial groups of order $p^3$ with exponent $p^2$, Des. Codes Cryptogr.,  \textbf{75} (2015), 237--242. 



\bibitem{Schur}
\emph{I.~Schur}, Zur theorie der einfach transitiven Permutationgruppen, S.-B. Preus Akad. Wiss. Phys.-Math. Kl., \textbf{18}, No.~20 (1933), 598--623.


\bibitem{Somma}
\emph{C.~Somma}, An infinite family of perfect codes in antipodal graphs, Rend. Mat. Appl., \textbf{3} (1983), 465--474.


\bibitem{Thas}
\emph{J.~A.~Thas}, Two infinite classes of perfect codes in metrically regular graphs, J. Combin. Theory Ser. B, \textbf{23} (1977), 236--238.



\bibitem{Wi}
\emph{H.~Wielandt}, Finite permutation groups, Academic Press, New York - London (1964).


\bibitem{Wilson}
\emph{R.~A.~Wilson}, The finite simple groups, Springer-Verlag London Limited (2009).

\bibitem{Winter}
\emph{D.~L.~Winter}, The automorphism group of an extraspecial $p$-group, Rocky Mountain J. Math., \textbf{2}, No.~2 (1972), 159--168.



\end{thebibliography}
\end{document}